\documentclass[a4paper,11pt]{amsart}
\usepackage{graphicx,amsmath,amssymb,amsthm,amscd,mathrsfs}

\DeclareMathOperator{\Hom}{Hom}

\DeclareMathOperator{\ord}{ord}

\newtheorem{thm}{Theorem}[section]
\newtheorem{df}[thm]{Definition}
\newtheorem{prop}[thm]{Proposition}
\newtheorem{lem}[thm]{Lemma}
\newtheorem{cor}[thm]{Corollary}
\newtheorem{rem}[thm]{Remark}

\begin{document}

\title{Local Fourier transform and blowing up}
\date{}
\author{Kazuki Hiroe}
\thanks{The author is supported by JSPS Grant-in-Aid for Young Scientists (B)
Grant Number 26800072.}
\email{kazuki@josai.ac.jp}
\keywords{Local Fourier transform, blow up, Milnor number, Stokes structure, 
iterated torus knot}
\subjclass[2010]{14H20,14H50,34M25,34M35,34M40}
\address{Department of Mathematics, Josai University,\\ 
1-1 Keyakidai Sakado-shi Saitama 350-0295 JAPAN.}
\begin{abstract}
	We consider a resolution of ramified irregular singularities
	of meromorphic connections on a formal disk via 
	local Fourier transforms.
	A necessary and sufficient condition for an irreducible connection
	to have a resolution of the ramified singularity is determined
	as an analogy of the blowing up of plane curve singularities.
	We also relate the irregularity of Komatsu and Malgrange of 
	connections 
	to the intersection numbers and the Milnor numbers of plane curve germs.
	Finally, 
	we shall define an analogue of Puiseux characteristics for connections 
	and find an invariant of the family of connections 
	with the fixed Puiseux characteristic by means of the structure of 
	iterated torus knots 
	of the plane curve germs.
\end{abstract}
\maketitle
\section*{Introduction}
The Fourier-Laplace transform plays important roles in the theory of 
ordinary differential equations on the Riemann sphere.
The local analogy of the transform, say the local Fourier transform,
is introduced by Laumon \cite{Lau} in the $l$-adic setting, 
and by Bloch-Esnault \cite{BloEsn} 
and Garc\'ia Lop\'ez \cite{Gar} in the complex domain 
to study local structures of the image of 
Fourier transform of global differential equations.
There are many applications of this transform to the analytic theory of 
differential equations, for example, see the works of Mochizuki \cite{Moc},
Sabbah \cite{Sab2} 
and Hien-Sabbah\cite{HieSab} in which the local Fourier transform is 
successfully applied to study the Stokes structure of differential equations.

In this paper, we ask a question: 
Is there a resolution of ramified irregular singularities of differential
equations 
via local Fourier transforms?
That is,  we shall consider an analogy of resolution of singularities of 
plane curve germs.
In fact, a similarity between local Fourier transforms and the blowing 
up of plane curves is pointed out by Sabbah who successfully uses
the blowing up technique
to calculate an explicit formula of local Fourier transforms in \cite{Sab}
after the work of Roucairol in \cite{Rou}.

To state our main theorems, we recall some definitions which are 
explained in detail in the latter sections.
Let $K$ be an algebraically closed field of characteristic zero. 
For a positive integer $q$ and $f\in K(\!(x^{\frac{1}{q}})\!)$
with $-p/q=\ord(f)$, 
let us define 
$E_{f,q}=(V,\nabla)$, a connection over $K( \!(x)\!)$, as follows. 
Regard 
$V=K(\!(x^{\frac{1}{q}})\!)$ as a $K(\!(x)\!)$-vector space
and define $\nabla(v)=(\frac{d}{dx}+x^{-1}f)v$ for $v\in V$.
To an irreducible $E_{f,q}$, we associate a plane curve germ,
	\[
		C_{f,q}(x,y)=\prod_{k=1}^{q}\left(y-
		\frac{1}{f_{k}(x^{\frac{1}{q}})}\right),
	\]
	where $f_{k}(x^{\frac{1}{q}})=
	f(\zeta_{q}^{k}
	x^{\frac{1}{q}})$ and $\zeta_{q}$ is a primitive $q$-th root of 
	unity.
Then the intersection numbers $I(\ ,\ )$ and Milnor numbers $\mu$ 
of curve germs can be 
written by the irregularities of connections as follows.
\begin{thm}[Theorem \ref{diffinv}]
	Let $E_{f,q}=(V,\nabla),\,E_{g,q'}=(W,\nabla')$ 
	be irreducible $K(\!(x)\!)$-connections. Set
	$-p/q=\ord(f),\,-p'/q'=\ord(g)$.
	If $E_{f,q}\not\cong E_{g,q'}$, then
	\[
		I\left(C_{f,q},C_{g,q'}\right)=
		pq'+p'q-\mathrm{Irr}(\mathrm{Hom}_{K(\!(x)\!)}(V,W)).
	\]
\end{thm}
\begin{thm}[Theorem \ref{milnumb}]
Let $E_{f,q}$ be an irreducible $K(\!(x)\!)$-connection with
	$\ord(f)=-p/q$. Then the Milnor number $\mu$ of 
	the associated curve $C_{f,q}$ is
	\[
		\mu=(2p-1)(q-1)
		-\mathrm{Irr}(\mathrm{End}_{K(\!(x)\!)}(V)).
	\]
\end{thm}
To an irreducible $E_{f,q}$, 
we associate a sequence of integers as follows.
Let us write $f(x^{\frac{1}{q}})=
a_{n}x^{\frac{n}{q}}+a_{n+1}x^{\frac{n+1}{q}}+\cdots$.
Define 
\[
	-\beta_{1}=\mathrm{min}\{i\mid a_{i}\neq 0,\, q\not| i\},\quad
	e_{1}=\mathrm{gcd}(q,\beta_{1}).
\]
Also define
\[
	-\beta_{k}=\mathrm{min}\{i\mid a_{i}\neq 0,\, e_{k-1}\not| i\},\quad
	e_{k}=\mathrm{gcd}(e_{k-1},\beta_{k}),
\]
inductively till we reach $g$ with $e_{g}=1$.
Then we call the sequence of the integers
\[
	(q,p; \beta_{1},\ldots,\beta_{g}),
\]
the dual Puiseux characteristic of $E_{f,q}$.

The explicit formula of the local Fourier transforms is 
independently obtained by 
two authors, Fang  in \cite{Fan} and Sabbah in \cite{Sab}, from the 
different point of view.  Fang uses an algebraic computation and 
Sabbah uses a technique of the blowing up of curves.
After their works Graham-Squire gave an simple description of their 
formula in \cite{Gra}.
We shall see that Graham-Squire's description leads us to 
a formula which may connect Fang's and Sabbah's different approaches.
Namely, we show that the local Fourier transform 
of $E_{f,q}$ can be seen as the blowing up of the
associated curve $C_{f,q}$, see Proposition \ref{fourierblowup}.
Furthermore as an application of this formula, 
the blowing up technique can be useful to show the following 
necessary and sufficient condition for the existence of a resolution of 
ramified irregular singularities of irreducible connections via 
local Fourier transforms.
\begin{thm}[Theorem \ref{main}]
Let $E_{f,q}$ be an irreducible connection with the dual Puiseux 
characteristic $(q,p;\beta_{1},\ldots,\beta_{g})$. Then 
we can reduce $E_{f,q}$ to a rank 1 connection by a finite iteration of 
local Fourier
transforms and additions if and only if we have
\[
	e_{i-1}\equiv \pm e_{i}\ (\mathrm{mod}\,\beta_{i})
\]
for all $i=1,\ldots,g$. Here $e_{0}=q$.
\end{thm}

In the final section, we shall moreover consider the following problem.
We take $K$ as the field of complex numbers $\mathbb{C}$.
Let us fix a dual Puiseux characteristic $(q,p;\beta_{1},\ldots,\beta_{g})$ and
consider a family 
\begin{multline*}
	\mathcal{E}=\{E_{f,q}\colon \text{irreducible connection}\,
		|\, \\
		E_{f,q}\text{ has the  
		dual Puiseux characteristic }
		(q,p;\beta_{1},\ldots,\beta_{g})
	\}.
\end{multline*}
We look for an invariant of this family whose elements are not isomorphic 
as connections in general.
For instance, it is well known that if two
plane curve germs have the same Puiseux characteristic, then the knot 
structures 
of these curves around the singular point are isotopic.
Namely the Puiseux characteristic gives an topological invariant of 
plane curve germs.
The aim of this section is to look for the analogy for connections.

Let us  fix an element 
$E_{f,q}\in \mathcal{E}$ and define $\tilde{f}(x^{\frac{1}{q}})=
\sum_{i=1}^{g}a_{\beta_{i}}x^{-\frac{\beta_{i}}{q}}$ where we write 
$f(x^{\frac{1}{q}})=a_{p}x^{-\frac{p}{q}}+a_{p-1}x^{-\frac{p-1}{q}}+\cdots$.
Also define $\tilde{f}_{i}(x^{\frac{1}{q}})
=\tilde{f}(\zeta_{q}^{i}x^{\frac{1}{q}})$ for $i=1,\ldots,q$.
If $x$ moves in a small circle 
$S_{\eta}=\{z\in \mathbb{C}\mid 
|z|=\eta\}$, 
the order of sizes of $\mathrm{Re}(\tilde{f}_{i}(x^{\frac{1}{q}}))$
for $i=0,\ldots,q-1$ 
change according to the argument of $x$.
Namely, we have $\mathrm{Re}(\tilde{f}_{i}(x^{\frac{1}{q}}))<
\mathrm{Re}(\tilde{f}_{j}(x^{\frac{1}{q}}))$ for an argument,
 $\mathrm{Re}(\tilde{f}_{i}(x^{\frac{1}{q}}))>
\mathrm{Re}(\tilde{f}_{j}(x^{\frac{1}{q}}))$ 
for another argument and there also are some arguments for which
these are incomparable.
This is one of the reasons why the Stokes phenomenon happens.
Thus to understand the Stokes phenomenon of the connections 
over $\mathbb{C}(\!\{x\}\!)$ formally isomorphic to $E_{f,q}$,
we study the closed curve
\[
	\mathrm{St}=\left\{(x,y)\,\middle|\, x\in S_{\eta},\,
	y=\mathrm{Re}(\tilde{f}(x^{\frac{1}{q}}))
	\right\}.
\]
Then Theorem \ref{iterated braid} and Corollary \ref{conjugate} 
show that the curve $\mathrm{St}$ has 
an invariant which depends only on the dual Puiseux characteristic and 
is independent of $E_{f,q}\in \mathcal{E}$.
The invariant is obtained from the structure of iterated torus knot of 
the associated curve germ $C_{\tilde{f},q}(x,y)$.

The space of the Stokes matrices 
of $\mathbb{C}(\!\{x\}\!)$-connections
which is formally isomorphic to $E_{f,q}$ is determined by the curve 
$\mathrm{St}$ (see Theorem \ref{riemann-hilbert} for instance).
Thus we may say that our theorem also gives an `topological' invariant of 
wild fundamental groupoid \cite{MarRam} and wild character varieties 
\cite{Boa}.

\section{Singularities of plane curve germs}
In this section we give basic definitions and facts
on singularities of plane curve germs, 
which are found in standard references
\cite{BrKn, Hef, Wal} for example.
Let $K$ be an algebraically closed field of characteristic zero.
Let $K[x], K[\![x]\!]$ and  $K(\!(x)\!)$
denote the polynomial ring, the ring of formal power series and 
the field of formal Laurent series. 
For $f(x)=a_{n}x^{n}+a_{n+1}x^{n+1}+\cdots\in K(\!(x)\!)$, 
we call the lowest exponent with nonzero coefficient the {\em order} of $f$ 
and denote by $\mathrm{ord}_{x}(f)$, i.e.,
\[
	\mathrm{ord}_{x}(f)=\mathrm{min}\{i\mid a_{i}\neq 0\}.
\]

Similarly the multi-variable analogue,  
$K[x,y], K[\![x,y]\!]$ are defined. 
We can decompose $f(x,y)\in K[\![x,y]\!]$ as 
the sum of homogeneous terms,
\[
	f(x,y)=\cdots+f_{k}(x,y)+f_{k+1}(x,y)+\cdots,
\]
where $f_{k}(x,y)\in K[x,y]$ are homogeneous polynomials of degree $k$.
The least integer $k_0$ 
such that $f_{k}(x,y)\neq 0$ is called {\em multiplicity} of $f$.

\begin{df}\normalfont A {\em plane curve germ} is the equivalence class of a 
non-invertible element $f$ of $K[\![x,y]\!]\backslash\{0\}$.
Here $f,g\in K[\![x,y]\!]$ are equivalent when there is a unit
$u\in K[\![x,y]\!]$ such that $f=ug$.
	A plane curve germ of multiplicity one is called {\em regular}.
	When the multiplicity is greater than one, the curve is called 
	{\em singular}.
\end{df}
\subsection{Good parametrizations}
Suppose that  $f(x,y)\in K[\![x,y]\!]\backslash\{0\}$ 
is {\em regular of order} $m>0$ with respect to $y$,
i.e., $f(0,y)\in K[\![y]\!]$ has the order $m$. 
Then the Weierstrass preparation theorem says that there exists an
unit $u\in K[\![x,y]\!]$ such that 
\[
	f(x,y)=u\left(y^{m}+\sum_{r=0}^{m-1}a_{r}(x)y^{r}\right)
\]
where $a_{r}(x)\in K[\![x[\!]$.

Then Puiseux's theorem tells us that $f$ can be decomposed as
\[
	f(x,y)=u\prod^{m}_{j=1}\left(y-g_{j}(x^{\frac{1}{m_{j}}})\right),
\]
where $g_{j}(t)\in K[\![t]\!]$. 
\begin{df}\normalfont
Let $f(x,y)$ be an irreducible element in $K[\![x,y]\!]$ 
	and regular of order $l>0$ with respect of $y$.
	Then we see that the equation $f(x,y)=0$ admits at least one solution
	of the form $y=\phi(x^{\frac{1}{m}})$ with $\phi(t)\in K[\![t]\!]$.
	Here we may assume 
	\[
		m=\mathrm{min}\left\{r\in \mathbb{N}\,\middle|\, 
	\phi \in K(\!(x^{\frac{1}{r}})\!) \right\}.
	\]
	Then the parametrization $x=t^{m}, y=\phi(t)$ of the curve germ is 
	called the {\em good parametrization}.
\end{df}
Conversely a good parametrization defines an irreducible 
curve germ as follows. Let $x=t^{m}, y=\phi(t)$ 
be a good parametrization and define 
\[
	f(x,y)=\prod_{i=1}^{m}\left(y-\phi(\zeta_{m}^{i}x^{\frac{1}{m}})
	\right),
\]
where $\zeta_{m}$ is a primitive $m$-th root of unity.

Let $x=t^{m}, y=\sum_{i\ge n}a_{i}t^{i}\ (a_{n}\neq 0)$ 
be a good parametrization. Here we may assume $n\ge m$ because 
if not, we can take another parametrization $y=u^{n}, 
x=\sum_{i\ge m}b_{i}u^{i}\ (b_{m}\neq 0)$ by
solving $u^{n}=\sum_{i\ge n}a_{i}t^{i}$.
Define $\beta_{1}$ to be the first exponent of $\sum_{i\ge n}a_{i}t^{i}$
which is indivisible by $m$ and $e_{1}$ to be the greatest common divisor
of $m$ and $\beta_{1}$, i.e.,
\[
	\beta_{1}=\mathrm{min}\{i\mid a_{i}\neq 0,\, m\!\not|\,i\},\quad 
	e_{1}=\mathrm{gcd}(m,\beta_{1}).
\]
Inductively define 
\[
	\beta_{k}=\mathrm{min}\{i\mid a_{i}\neq 0,\, e_{k-1}\!\not|\,i\},\quad 
	e_{k}=\mathrm{gcd}(e_{k-1},\beta_{k})
\]
till we reach $g$ with $e_{g}=1$.
\begin{df}\normalfont
	For the above good parametrization, the sequence of positive integers
	\[
		(m;\beta_{1},\ldots,\beta_{g})
	\]
	is called the {\em Puiseux characteristic} of the curve germ.
\end{df}
\subsection{Blowing up}
Let us recall the blowing up of the affine space $\mathbb{A}^{2}(K)$.
\begin{df}\normalfont
	Let us define a subspace of 
	$\mathbb{A}^{2}(K)\times \mathbb{P}^{1}(K)$ by
	\[
		T=\{(x,y, (\xi\colon \eta))\mid x\eta=y\xi\},
	\]
	where $(\xi\colon\eta)$ is the homogeneous coordinate of 
	$\mathbb{P}^{1}(K)$.
	Then the natural projection $\pi\colon T\rightarrow \mathbb{A}^{2}(K)$
	is called the {\em blowing up} of $\mathbb{A}^{2}(K)$ with the center
	$O=(0,0)$.
\end{df}

The projective line $\mathbb{P}^{1}(K)$ 
is covered by two open sets $U_{1}=\{(\xi\colon \eta)
\mid \xi\neq 0\}$ and $U_{2}=\{(\xi\colon \eta)\mid \eta\neq 0\}$ which
are isomorphic to $\mathbb{A}(K)$. 
Hence we can cover $T$ by $T_{1}=T\cap (\mathbb{A}^{2}(K)\times U_{1})$ and
$T_{2}=T\cap (\mathbb{A}^{2}(K)\times U_{2})$. 
Both $T_{1}$ and $T_{2}$ can be seen as $\mathbb{A}^{2}(K)$ by
\begin{align*}
	&\begin{array}{cccc}
	\rho_{1}\colon&T_{1}&\longrightarrow&\mathbb{A}^{2}(K)\\
		&(x,y,(\xi\colon\eta))&\longmapsto&
		(X,Y)=(x,\frac{\eta}{\xi})
	\end{array},\\
	&\begin{array}{cccc}
	\rho_{2}\colon&T_{2}&\longrightarrow&\mathbb{A}^{2}(K)\\
		&(x,y,(\xi\colon\eta))&\longmapsto&
		(X,Y)=(\frac{\xi}{\eta},y).
	\end{array}
\end{align*}
Thus the restrictions of $\pi$ on $T_{1}$ and $T_{2}$ give transformations
in $\mathbb{A}^{2}(K)$, say $\sigma_{1}=\pi\circ\rho_{1}^{-1}$ and $\sigma_{2}=
\pi\circ\rho_{2}^{-1}$. 
\begin{df}\normalfont Transformations in $\mathbb{A}^{2}(K)$ defined by
	\begin{align*}
		&\begin{array}{cccc}
		\sigma_{1}\colon&\mathbb{A}^{2}(K)&\longrightarrow
		&\mathbb{A}^{2}(K)\\
			&(x_{1},y_{1})&\longmapsto&(x,y)=(x_{1},x_{1}y_{1})
		\end{array},\\
		&\begin{array}{cccc}
		\sigma_{2}\colon&\mathbb{A}^{2}(K)&\longrightarrow
		&\mathbb{A}^{2}(K)\\
			&(x_{1},y_{1})&\longmapsto&(x,y)=(x_{1}y_{1},y_{1})
		\end{array},
	\end{align*}
	are called {\em quadratic transforms}.
	These induce homomorphisms
	\begin{align*}
		&\begin{array}{cccc}
	\sigma_{1}\colon&K[\![x,y]\!]&\longrightarrow
	&K[\![x_{1},y_{1}]\!]\\
			&x&\longmapsto&x_{1}\\
			&y&\longmapsto&x_{1}y_{1}
		\end{array},
		&\begin{array}{cccc}
	\sigma_{2}\colon&K[\![x,y]\!]&\longrightarrow
		&K[\![x_{1},y_{1}]\!]\\
			&x&\longmapsto&x_{1}y_{1}\\
			&y&\longmapsto&y_{1}
		\end{array}.
	\end{align*}
	These are called quadratic transforms as well.
\end{df}
If $f(x,y)$ has the multiplicity $k$, then $\sigma_{1}(f)(x_{1},y_{1})$ and 
$\sigma_{2}(f)(x_{1},y_{1})$ can
be divided by $x_{1}^{k}$ and $y_{1}^{k}$ respectively.
\begin{df}\normalfont
	Let $f(x,y)$ be a curve germ with the multiplicity $k$. Then
	$\sigma_{1}^{*}(f)(x_{1},y_{1})=
	\frac{1}{x_{1}^{k}}\sigma_{1}(f)(x_{1},y_{1})$ and 
	$\sigma_{2}^{*}(f)(x_{1},y_{1})=
	\frac{1}{y_{1}^{k}}\sigma_{2}(f)(x_{1},y_{1})$ 
	are called {\em strict transforms} of $f$.
\end{df}

Suppose that  an irreducible
curve germ $f$ has a good parametrization 
$x=t^{m},\ y=a_{n}t^{n}+a_{n+1}t^{n+1}+\cdots$ where $a_{n}\neq 0$ and $n\ge m$.
Then the strict transform $\sigma_{1}^{*}(f)(x_{1},y_{1})$ has the good
parametrization $x_{1}=t^{m},\ y_{1}=a_{n}t^{n-m}+a_{n+1}t^{n-m+1}+\cdots$.
Here we note that $\sigma_{1}(f)^{*}(0,0)\neq 0$, i.e., $\sigma_{1}^{*}$
is invertible if $n=m$. 
Thus we define
\[
	\sigma^{*}(f)(x_{1},y_{1})=\begin{cases}
		\sigma_{1}^{*}(f)(x_{1},y_{1})&\text{if }n>m,\\
		\sigma_{1}^{*}(f)(x_{1},y_{1}-a_{n})&\text{if }n=m,
	\end{cases}
\]
and call $\sigma^{*}(f)(x_{1},y_{1})$ the {\em blowing up} of 
$f$. 
If $f$ has a good parametrization 
$y=u^{n},\ x=b_{m}u^{m}+b_{m+1}u^{m+1}+\cdots$ 
where $b_{m}\neq 0$ and $m\ge n$, then we define 
\[
	\sigma^{*}(f)(x_{1},y_{1})=\begin{cases}
		\sigma_{2}^{*}(f)(x_{1},y_{1})&\text{if }m>n,\\
		\sigma_{2}^{*}(f)(x_{1}-b_{m},y_{1})&\text{if }m=n,
	\end{cases}
\]
similarly.

Let us see how the blowing up changes Puiseux 
characteristics of curve germs.
\begin{prop}[see Theorem 3.5.5 in \cite{Wal} for example]
	For an irreducible curve germ $f(x,y)$ with 
	the Puiseux characteristic $(m; \beta_{1},\ldots,\beta_{g})$,
	we can compute the Puiseux characteristic
	of $\sigma^{*}(f)$ as follows.
	\begin{enumerate}
		\item If $\beta_{1}>2m$,  
			\[
				(m; \beta_{1}-m,\ldots,\beta_{g}-m).
			\]
		\item If $\beta_{1}<2m$ and $(\beta_{1}-m)\! \not|\, m$, 
			\[
				(\beta_{1}-m; m,\beta_{2}-\beta_{1}+m,
				\ldots,\beta_{g}-\beta_{1}+m).
			\]
		\item If $(\beta_{1}-m)\,|\,m$,
			\[
				(\beta_{1}-m; \beta_{2}-\beta_{1}+m,\ldots,
				\beta_{g}-\beta_{1}+m).
			\]
	\end{enumerate}
\end{prop}

\subsection{Some invariants of curves}

\begin{df}\normalfont
	Let $f,g$ be curve germs. Then
	the integer
	\[
		I(f,g)=\mathrm{dim}_{K}K[\![x,y]\!]/
		\langle f,g\rangle
	\]
	is called the {\em intersection number} of $f$ and $g$. Here
	$\langle f,g\rangle$ is the ideal of $K[\![x,y]\!]$
	generated by $f,g$.
\end{df}
\begin{df}\normalfont
	Let $f(x,y)$ be a curve germ. Then the integer 
	\[
		\mu=I\left(\frac{\partial f}{\partial x},\frac{\partial f}
		{\partial y}\right)
	\]
	is called the {\em Milnor number} of $f$.
\end{df}

These integers can be computed from good parametrizations and 
Puiseux characteristics as follows, see the section 4 in \cite{Wal} and 
the section 7.4 in \cite{Hef}. 
\begin{lem}\label{invariants}
	Let $f(x,y)$ be an irreducible curve germ with a good 
	parametrization 
	$x=t^{m},\ y=\phi(t)=a_{n}t^{n}+a_{n+1}t^{n+1}+\cdots$  $(n\ge m)$,
	and the Puiseux characteristic $(m; \beta_{1},\ldots,\beta_{g})$.
	\begin{enumerate}
		\item For a curve germ $g(x,y)\neq f(x,y)$, the intersection
			number $I(f,g)$ is equal to the order
			of $g(t^{m},\phi(t))$.
		\item The Milnor number of $f$ is 
			\[
				\mu=\sum_{i=1}^{g}(e_{i-1}-e_{i})(\beta_{i}-1).
			\]
	\end{enumerate}
\end{lem}
\begin{lem}[see COROLLARY 7.16 and THEOREM 7.18 in \cite{Hef} for example]
	\label{invariants2}
	Let $f(x,y)$ be an irreducible curve germ with $n=I(f,x)$ 
	and $m=I(f,y)$. Then  we have
	\begin{align*}
		\mu&=I\left(f,\frac{\partial f}{\partial x}\right)-m+1,&
		\mu&=I\left(f,\frac{\partial f}{\partial y}\right)-n+1.
	\end{align*}

\end{lem}

\section{Formal meromorphic connections on a disk}
In this section we recall basic definitions and facts of 
formal meromorphic connections on a disk.
\begin{df}\normalfont
	Let $V$ be a finite dimensional vector space over 
	$K(\!(x)\!)$.  A {\em connection}
	on $V$ is a $K$-linear
	map $\nabla\colon V\rightarrow V$ satisfying the Leibniz rule
	\[
		\nabla(fv)=f\nabla(v)+\frac{df}{dx}\nabla(v)
	\]
	for $f\in K(\!(x)\!)$ and $v\in V$. 
	We call the pair $(V,\nabla)$ the $K(\!(x)\!)$-connection 
	shortly.
	Sometimes we write $(V,\nabla_{x})$ to emphasize 
	the variable $x$.
\end{df}
The {\em rank} of $(V,\nabla)$ is the dimension of $V$ as the
$K(\!(x)\!)$-vector space. 
We say that $(V,\nabla)$ is {\em irreducible} if $V$ has no proper nontrivial
$K(\!(x)\!)$-subspace $W$ such that $\nabla(W)\subset W$.
Morphisms between connections $(V_{1},\nabla_{1})$ and $(V_{2},\nabla_{2})$
are $K( \!(x)\!)$-linear maps $\phi\colon V_{1}\rightarrow V_{2}$ 
satisfying $\phi\nabla_{1}=\nabla_{2}\phi$.

\subsection{Indecomposable decompositions of connections}
Let us give a quick review of 
indecomposable decompositions of connections based on
the works of Hukuhara, Turrittin, Levelt, Balser-Jurkat-Lutz, Babbitt-
Varadarajan and so on, \cite{Huk,Tur,Lev,BJL,BabVar}.
We adopt the descriptions in 
\cite{Gra, Sab}.

For a positive integer $q$ and $f\in K(\!(x^{-\frac{1}{q}})\!)$, 
let us define 
$E_{f,q}=(V,\nabla)$, a connection over $K( \!(x)\!)$, as follows. 
Regard 
$V=K(\!(x^{\frac{1}{q}})\!)$ as a $K(\!(x)\!)$-vector space
and define $\nabla(v)=(\frac{d}{dx}+x^{-1}f)v$ for $v\in V$.
The irreducibility and isomorphic classe of $E_{f,q}$ are determined as follows
(see the section 3 in \cite{Sab} for example).
If $E_{f,q}$ and $E_{g,q}$ are isomorphic, then
			there exists an integer $0\le r\le q-1$ such
			that 
			\[
				f(x^{\frac{1}{q}})-
				g(\zeta_{q}^{r}x^{\frac{1}{q}})
				\in 
				R_{q}(x)=
				K( \!(x^{\frac{1}{q}})\!)/
				\left(x^{\frac{1}{q}}K[ \![x^{\frac{1}{q}}]\!]
					+\frac{1}{q}\mathbb{Z}
				\right)
			\]
Also the converse is true. 
Let us define $R_{q}^{o}(x)$ as the set of $f\in R_{q}(x)$ that cannot 
be represented by elements of $K(\!(x^{\frac{1}{r}})\!)$ for any
$0<r<q$. Then 
the connection $E_{f,q}$ is irreducible if and only if the image of 
$f$ in $R_{q}^{o}$.
\begin{prop}[Hukuhara-Turrittin-Levelt decomposition]\label{HTL}
	Every $(V,\nabla)$ decomposes as 
	\[
		(V,\nabla)\cong \bigoplus_{i}(E_{f_{i},q_{i}}\otimes J_{m_{i}})
	\]
	where $f_{i}\in R_{q_{i}}^{o}(x)$ and $J_{m}=(
	\mathbb{C}(\!(x)\!)^{\oplus m},\frac{d}{dx}+x^{-1}N_{m})$ with
	the nilpotent Jordan block $N_{m}$ of size $m$. 
\end{prop}
\subsection{Local Fourier transforms}
The local Fourier transform is 
introduced by Laumon, Bloch-Esnault and Garc\'ia L\'opez \cite{Lau,BloEsn,Gar}
to analyze formal local structures of the Fourier transform of meromorphic 
connections on $\mathbb{P}^{1}$. 
In this paper, we consider the local Fourier transform only for $E_{f,q}$
following Proposition 3.7, 3.9 and 3.12 in \cite{BloEsn}
and refer to original papers for general definitions and properties.
\begin{df}\normalfont 
	Let $z,\hat{z}$ be indeterminates and set $\zeta=\frac{1}{z}, 
	\hat{\zeta}=\frac{1}{\hat{z}}.$
	\begin{enumerate}
		\item Let $f\in R_{q}^{o}(z)$ and 
			$f\neq 0$. Set $E_{f,q}=(V,\nabla_{z})$.
			The connection $\mathcal{F}^{(0,\infty)}(E_{f,q})
			=(V,\hat{\nabla}_{\hat{\zeta}})$ over
			$K(\!(\hat{\zeta})\!)$ is defined by 
			$K$-linear operators on $V$,
			\[
			\hat{\zeta}=-\nabla_{z}^{-1}\colon
				V\rightarrow V,\quad
				\hat{\nabla}_{\hat{\zeta}}=
				-\hat{\zeta}^{-2}z\colon V
				\rightarrow V.
			\]
		\item Let $f\in R_{q}^{o}(\zeta)$, $\ord(f)=-p/q$, 
			$f\neq 0$. 
			Set $E_{f,q}=(V,\nabla_{\zeta})$ and 
			suppose that $p<q$. 
			Then the connection 
			$\mathcal{F}^{(\infty,0)}(E_{f,g})
			=(V,\hat{\nabla}_{\hat{z}})$ over
			$K(\!(\hat{z})\!)$ is obtained by 
			$K$-linear operators on $V$,
			\[
				\hat{z}=\zeta^{2}\nabla_{\zeta}\colon
				V\rightarrow V,\quad
				\hat{\nabla}_{\hat{z}}=
				-\zeta^{-1}\colon V
				\rightarrow V.
			\]
		\item Let $f\in R_{q}^{o}(\zeta)$, $\ord(f)=-p/q$, 
			$f\neq 0$.
			Set $E_{f,q}=(V,\nabla_{\zeta})$ and 
			suppose that $p>q$. 
			Then the connection 
			$\mathcal{F}^{(\infty,\infty)}(E_{f,g})
			=(V,\hat{\nabla}_{\hat{\zeta}})$ over
			$K(\!(\hat{\zeta})\!)$ is obtained by 
			$K$-linear operators on $V$,
			\[
				\hat{z}=\zeta^{2}\nabla_{\zeta}\colon
				V\rightarrow V,\quad
				\hat{\zeta}^{2}\hat{\nabla}_{\hat{\zeta}}=
				-\zeta^{-1}\colon V
				\rightarrow V.
			\]
	\end{enumerate}
\end{df}

The following theorem for the explicit structures of 
local Fourier transforms $\mathcal{F}^{(*,*)}(E_{f,q})$ is due to Fang and 
Sabbah. We adopt the formulation given by Graham-Squire in \cite{Gra} who gave a 
simple proof of the theorem.
\begin{thm}[J.~Fang \cite{Fan} and C.~Sabbah \cite{Sab}]\label{locfourier}
	\begin{enumerate}
		\item Let $f\in R_{q}^{o}(z)$, 
			$\ord(f)=-p/q$ and $f\neq 0$. 
			Then
			\[
				\mathcal{F}^{(0,\infty)}(E_{f,q})\cong
				E_{g,p+q},
			\]
			where $g\in R_{p+q}^{o}(\hat{\zeta})$ is determined by
			\begin{align*}
				f&=-z\hat{z},& 
				g&=f+\frac{p}{2(p+q)}.
			\end{align*}
		\item Let $f\in R_{q}^{o}(\zeta)$, 
			$\ord(f)=-p/q$ and $f\neq 0$. Suppose that 
			$p<q$.
			Then
			\[
				\mathcal{F}^{(\infty,0)}(E_{f,q})\cong
				E_{g,q-p},
			\]
			where $g\in R_{q-p}^{o}(\hat{z})$ is determined by 
			\begin{align*}
				f&=z\hat{z},& 
				g&=-f+\frac{p}{2(q-p)}.
			\end{align*}
 		\item Let $f\in R_{q}^{o}(\zeta)$, 
			$\ord(f)=-p/q$ and $f\neq 0$. Suppose that 
			$p>q$.
			Then
			\[
				\mathcal{F}^{(\infty,\infty)}(E_{f,q})\cong
				E_{g,p-q},
			\]
			where $g\in R_{p-q}^{o}(\hat{\zeta})$ is determined by
			\begin{align*}
				f&=z\hat{z},& 
				g&=-f+\frac{p}{2(p-q)}.
			\end{align*}
	\end{enumerate}
\end{thm}

By the above theorem, 
we can define the inversion of local Fourier transforms as follows.
Let $g\in R_{q}^{o}(\hat{\zeta})$, $\ord(g)=-p/q$ and $g\neq 0$. 
Suppose that $p<q$. 
Then we define  $\left(\mathcal{F}^{(0,\infty)}\right)^{-1}
(E_{g,q})$ as  
$E_{f,q-p}$  where $f\in R_{q-p}^{o}(z)$ is determined by  
\begin{align*}
	g-\frac{p}{2q}&=-z\hat{z},& 
	g&=f+\frac{p}{2q}.
\end{align*}
Then we have 
\begin{align*}
\mathcal{F}^{(0,\infty)}\left(\left( \mathcal{F}^{(0,\infty)}
\right)^{-1}(E_{g,q})\right)&\cong E_{g,q},& 
\left(\mathcal{F}^{(0,\infty)}\right)^{-1}\left(\mathcal{F}^{(0,\infty)}
(E_{f,q})\right)&\cong E_{f,q}.
\end{align*}
Similarly we can define $(\mathcal{F}^{(\infty,0)})^{-1}(E_{g,q})$. 
If $p>q$, then $(\mathcal{F}^{(\infty,\infty)})^{-1}(E_{g,q})$ is defined
as well.

\subsection{Irregularity}
For a $K(\!(x)\!)$-connection $(V,\nabla)$, let us fix a basis
and identify $V\cong K(\!(x)\!)^{\oplus n}$. Then we can write
$\nabla=\frac{d}{dx}+A(x)$ with $A(x)\in M(n,K(\!(x)\!))$.
Moreover we can choose a suitable basis 
so that $A(x)\in M(n,K[x^{-1}])$,
see \cite{BJL} for example.  We call $A(x)\in M(n,K[x^{-1}])$ 
the {\em normalized matrix} of $(V,\nabla)$.

Let us take $K$ as the field of complex numbers $\mathbb{C}$
and $\mathbb{C}(\!\{x\}\!)$ denote the field of meromorphic functions 
near $0$. 
The irregularity defined below measures the difference 
between formal and convergent solutions of $\frac{d}{dx}+A(x)$.
\begin{df}[see H.~Komatsu \cite{Kom} and B.~Malgrange \cite{Mal}]\normalfont
	Let $(V,\nabla)$ be a $\mathbb{C}(\!(x)\!)$-connection. 
	The {\em irregularity} of 
	$(V,\nabla)$ is 
	\[
		\mathrm{Irr}(V,\nabla)=\chi\left(\frac{d}{dx}+A(x),\,
		\mathbb{C}(\!(x)\!)^{\oplus\, n}\right)-
		\chi\left(\frac{d}{dx}+A(x),\,
		\mathbb{C}(\!\{x\}\!)^{\oplus\,n}\right).
	\]
	Here $\chi(\Phi,V)$ is the {\em index} of the $\mathbb{C}$-linear map
	$\Phi\colon V\rightarrow V$, i.e.,
	\[
		\chi(\Phi,V)=\dim_{\mathbb{C}}\mathrm{Ker\,}\Phi-
		\dim_{\mathbb{C}}\mathrm{Coker}\Phi.
	\]
\end{df}
It is known that $\mathrm{Irr}(V,\nabla)$ is independent from choices of 
normalized matrices $A(x)$. Moreover if we decompose
	\[
		(V,\nabla)\cong \bigoplus_{i}(E_{f_{i},q_{i}}\otimes J_{m_{i}})
	\]
	as Proposition \ref{HTL}, the irregularity can be written by 
	\[
		\mathrm{Irr}( (V,\nabla))=-\sum_{i}ord(f_{i}).
	\]
	Thus not only for $\mathbb{C}$ but also general $K$, 
	we can define the irregularity by the above formula.

\section{Formal meromorphic connections and associated curves}
In this section we shall define curve germs associated to irreducible 
formal meromorphic connections. Then intersection numbers and Milnor
numbers of these curves will be written by the irregularities of the 
connections. Next we shall see the relationship between
the local Fourier transforms of connections and the blowing up of the curve
germs.
Finally we shall determine a necessary and sufficient condition 
for an irreducible formal 
connection to have a resolution of ramified irregular 
singularities via local Fourier
transforms.

\subsection{Associated curves}
Let us take $f\in K(\!(x^{\frac{1}{q}})\!)$ so that the image is in 
$R_{q}^{o}\backslash\{0\}$.
Then the curve germ associated to the irreducible $E_{f,q}$ is defined as follows.
\begin{df}\normalfont
	The {\em associated curve germ} of an irreducible 
	$K(\!(x)\!)$-connection $E_{f,q}$ is 
	\[
		C_{f,q}(x,y)=\prod_{i=1}^{q}\left(y-
		\frac{1}{f_{i}(x^{\frac{1}{q}})}\right),
	\]
	where $f_{i}(x^{\frac{1}{q}})=
	f(\zeta_{q}^{i}
	x^{\frac{1}{q}})$.
\end{df}

To the above $E_{f,q}$, 
we associate a sequence of integers as an analogy of the Puiseux
characteristic of curves. 
Let us write $f(x^{\frac{1}{q}})=
a_{n}x^{\frac{n}{q}}+a_{n+1}x^{\frac{n+1}{q}}+\cdots$.
Define 
\[
	-\beta_{1}=\mathrm{min}\{i\mid a_{i}\neq 0,\, q\not| i\},\quad
	e_{1}=\mathrm{gcd}(q,\beta_{1}).
\]
Also define
\[
	-\beta_{k}=\mathrm{min}\{i\mid a_{i}\neq 0,\, e_{k-1}\not| i\},\quad
	e_{k}=\mathrm{gcd}(e_{k-1},\beta_{k}),
\]
inductively till we reach $g$ with $e_{g}=1$.
\begin{df}\normalfont
	Let $E_{f,q}$ be as above with $-p/q=\ord(f)$. 
	Then the sequence of the integers
	\[
		(q,p; \beta_{1},\ldots,\beta_{g})
	\]
	is called the {\em dual Puiseux characteristic} of $E_{f,q}$.
\end{df}
Let us compare the dual Puiseux characteristic of $E_{f,q}$ with the Puiseux
characteristic of $C_{f,q}$.
\begin{prop}\label{puiseux}
	Let $E_{f,q}$ and $C_{f,q}$ be as above and 
	\[
	(q,p;\beta_{1},\ldots,\beta_{g})
	\]
	the dual Puiseux characteristic of $E_{f,q}$.
	Then the Puiseux characteristic of $C_{f,q}$ is 
	\[
		(q;2p-\beta_{1},\ldots,2p-\beta_{g}),\quad\text{if }p\ge q,
	\]
	\[
		(p;p+q-\beta_{1},\ldots,p+q-\beta_{g}),\quad\text{if }
		p<q.
	\]
\end{prop}
To prove this proposition, we need some preparations.
Let $S\subset \mathbb{Z}_{\ge 0}$ be an additive 
semigroup including $0$ and $S_{0}$ a subset of $S$ such that 
if $s=s'+s''$ with $s\in S_{0}$ and $s',s''\in S$ then either
$s'=0$ or $s''=0$. Write $\mathcal{O}_{S}$ for the set of power series
$\sum_{r}a_{r}t^{r}$ such that $a_{r}=0$ for all $r\notin S$ and 
$\mathcal{O}^{*}_{S}$ for the subset satisfying the further condition
$a_{r}\neq 0$ for all $r\in S_{0}$.
\begin{lem}[cf. Lemma 3.5.4 in \cite{Wal}]\label{wall}\quad\\
	\begin{enumerate}
		\item   Let $(t\alpha(t))^{m}=t^{m}\gamma(t)$ 
			with $m\in \mathbb{Z}$, 
			$\alpha(t),\gamma(t)\in
			K[\![t]\!]$, $\alpha(0)\neq 0$.
			Then $\alpha\in \mathcal{O}_{S}$ if and only if 
			$\gamma\in\mathcal{O}_{S}$, 
			and $\alpha\in\mathcal{O}^{*}_{S}$ 
			if and only if $\gamma\in \mathcal{O}^{*}_{S}$
		\item Let $\alpha\in K[\![t]\!]$ with 
			$\alpha(0)\neq 0$ and let $\beta\in 
			K[\![t]\!]$ be such that $t=u\beta(u)$
			solves $u=t\alpha(t)$. Then $\alpha\in \mathcal{O}_{S}$
			if and only if $\beta\in \mathcal{O}_{S}$  and 
			$\alpha\in\mathcal{O}^{*}_{S}$ if and only if 
			$\beta\in\mathcal{O}^{*}_{S}$.
		\end{enumerate}
\end{lem}
\begin{proof}	
	If we replace  the condition
	$m\in \mathbb{Z}$ in $(1)$ with $m\in \mathbb{N}$, 
	then this is nothing but Lemma 3.4.5 in Wall's book 
	\cite{Wal}. Thus we show only the case $m=-1$ in $(1)$.
	Although this follows from the same argument of the lemma in the 
	Wall's book, we give 
	a proof for the completeness of the paper.
	Write $\alpha(t)=\sum_{r=0}^{\infty}a_{r}t_{r}$ with $a_{0}\neq 0$
	and $\gamma(t)=\sum_{r=0}^{\infty}\gamma_{r}t^{r}$. 
	Then
	\[
		\sum_{r=0}^{\infty}\gamma_{r}t^{r}=
		\left(\sum_{r=0}^{\infty}a_{r}t_{r}\right)^{-1}.
	\]
	We may assume that $a_{0}=1$. Then
	\begin{align*}
		\gamma_{0}&=1,\\
		\gamma_{1}+a_{1}\gamma_{0}&=0,\\
		\cdots\\
		\gamma_{k}+a_{1}\gamma_{k-1}+\cdots+a_{k}\gamma_{0}&=0,\\
		\cdots.
	\end{align*}
	Thus $\gamma_{r}$ are  
	linear combinations of $a_{r_{1}}\cdots a_{r_{m}}$ with 
	$r=r_{1}+\cdots +r_{m}$. If $\gamma_{r}\neq 0$, then 
	there exist $r_{1},\ldots,r_{m}$ such that $r_{1}+\cdots +r_{m}=r$
	and $a_{r_{1}}\cdots a_{r_{m}}\neq 0$. Equivalently $r_{1},\ldots,
	r_{m}\in S$. Since $S$ is a semigroup, this means $r\in S$.
	Conversely, suppose $\gamma\in \mathcal{O}_{S}$ and that for each
	$r<k$ with $r\notin S$ we have $a_{r}=0$. If $k\notin S$, 
	$\gamma_{k}=na_{k}$ with a nonzero integer $n$. Thus $a_{k}=0$
	and it follows by induction on $k$ that $\alpha\in\mathcal{O}_{S}$.
	If further $p\in S_{0}$, then $p$ can not decompose by elements in $S$.
	Thus $\gamma_p=na_{p}$ with a nonzero integer $n$. Thus
	indeed $\gamma_{p}\neq 0$ if and only if $a_{p}\neq 0$.
\end{proof}
\begin{proof}[Proof of Proposition \ref{puiseux}]
	We trace the argument of Theorem 3.5.5 in \cite{Wal}.
	Set $x=t^{q}$ and $f(x^{\frac{1}{q}})=t^{-p}\alpha(t)$ with
	$\alpha(t)\in K[\![t]\!],\ \alpha(0)\neq 0$.
	Then the associated curve $C_{f,q}$ has a good parametrization
	$x=t^{q},\,y=t^{p}\alpha(t)^{-1}$.
	Let 
	\[
		S=\{r\in\mathbb{Z}\mid \text{for some }q\ge 1,\, r\ge
		p-\beta_{q}\text{ and }e_{q}|r
		\},
	\]
	and $S_{0}=\{p-\beta_{q}\mid q\ge 1\}$.
	Then from the hypothesis,
	we have $\alpha\in\mathcal{O}^{*}_{S}$ which shows $\alpha^{-1}\in
	\mathcal{O}^{*}_{S}$ by Lemma \ref{wall}.
	This shows the proposition in the case $p\ge q$.
	If $p<q$, we can put $y=t^{p}\alpha(t)^{-1}=(t\beta(t))^{p}$ with
	$\beta\in \mathcal{O}^{*}_{S}$ by Lemma \ref{wall}.
	Set $u=t\beta(t)$, so that $y=u^{p}$. We can write $t=u\gamma(u)$ with
	$\gamma\in\mathcal{O}^{*}_{S}$. Thus $x=t^{q}=(u\gamma(u))^{q}=
	u^{q}(\delta(u))$ with $\delta(u)\in\mathcal{O}^{*}_{S}$ which
	completes the proof.
\end{proof}
\subsection{Irregularity of connections and curve invariants}
Let us see that the irregularity of connections relates some curve invariants,
intersection numbers and Milnor numbers, of associated curve germs.
\begin{thm}\label{diffinv}
	Let $E_{f,q}=(V,\nabla),\,E_{g,q'}=(W,\nabla')$ 
	be irreducible $K(\!(x)\!)$-connections. Set
	$-p/q=\ord(f),\,-p'/q'=\ord(g)$.
	If $E_{f,q}\not\cong E_{g,q'}$, then
	\[
		I\left(C_{f,q},C_{g,q'}\right)=
		pq'+p'q-\mathrm{Irr}(\mathrm{Hom}_{K(\!(x)\!)}(V,W)).
	\]
	Here $\Hom_{K(\!(x)\!)}(V,W)$ can be naturally seen as a 
	$K(\!(x)\!)$-connection through
	the actions of $\nabla$ and $\nabla'$. Namely the connection 
	$\nabla^{''}$ on $\Hom_{K(\!(x)\!)}(V,W)$ is defined 
	by
	\[
		\nabla^{''}(\phi)(v)=\nabla'(\phi(v))-\phi(\nabla(v))
	\]
	for $\phi\in \Hom_{K(\!(x)\!)}(V,W)$ and $v\in V$.
	Similarly we have
	\[
		I\left(C_{f,q},\frac{\partial}{\partial y}
			C_{f,q}\right)=
		2p(q-1)-\mathrm{Irr}(\mathrm{End}_{K(\!(x)\!)}(V)).
	\]
\end{thm}
\begin{proof}
	The associated curve germ $C_{f,q}$ has a
	good parametrization, $x=t^{q},\,y=\alpha(t)=
	\frac{1}{f(x^{\frac{1}{q}})}$. Thus
	\begin{align*}
		I\left(C_{f,q},C_{g,q'}\right)&=
		\mathrm{ord}_{t}C_{g,q'}(t^{q},\alpha(t))\\
		&=\mathrm{ord}_{t}\prod_{i=1}^{q'}\left(\alpha(t)-
		\frac{1}{g_{i}(t^{\frac{q}{q'}})}\right)\\
		&=q\cdot\mathrm{ord}_{x}\prod_{i=1}^{q'}\left(
		\frac{1}{f(x^{\frac{1}{q}})}-
		\frac{1}{g(x^{\frac{1}{q'}})}
		\right)\\
		&=q\cdot\mathrm{ord}_{x}\prod_{i=1}^{q'}\left(
		\frac{g_{i}-f}{f g_{i}}
		\right)\\
		&=pq'+p'q+q\cdot\mathrm{ord}_{x}\prod_{i=1}^{q'}
		(g_{i}-f).
	\end{align*}
	Let us note that the intersection number does not depend on good 
	parametrizations, $x=t^{q}, y=1/ f_{j}$, $j=0,
	\ldots,q-1$. Thus 
	\[
	\mathrm{ord}_{x}\prod_{i=1}^{q'}
	(g_{i}-f_{j})=\mathrm{ord}_{x}\prod_{i=1}^{q'}
	(g_{i}-f_{j'})
	\]
	for $1\le j,j'\le q$.
	On the other hand, we have 
	\[
		\mathrm{Irr}(\mathrm{Hom}_{K(\!(x)\!)}(V,W))=
		-\mathrm{ord}_{x}\prod_{i=1}^{q'}\prod_{j=1}^{q}
		(g_{i}-f_{j}).
	\]
	Combining these equations, we have the required one.

	Let us see the second assertion.
	Since $C_{f,q}=\prod_{i=1}^{q}(y-1/ f_{i})$,
	we have $\frac{\partial}{\partial y}C_{f,q}=
	\sum_{i=1}^{q}\prod_{j\neq i}(y-1/f_{j})$.
	Thus we can show  
	\[
	I\left(C_{f,q},\frac{\partial}{\partial y}
			C_{f,q}\right)=
			\ord_{x}\prod_{\substack{1\le i,j\le q\\
			i\neq j}}\left(\frac{1}{f_{i}}-
			\frac{1}{f_{j}}
			\right),
	\]
	as above. Also recall
	\[
		\mathrm{Irr}(\mathrm{End}_{K(\!(x)\!)}(V))=
		-\ord_{x}\prod_{\substack{1\le i,j\le q\\
		i\neq j}}(f_{i}-f_{j}).
	\]
	These equations show the required one as above.
\end{proof}
\begin{thm}\label{milnumb}
	Let $E_{f,q}$ be an irreducible $K(\!(x)\!)$-connection with
	$\ord(f)=-p/q$. Then the Milnor number $\mu$ of 
	the associated curve $C_{f,q}$ is
	\[
		\mu=(2p-1)(q-1)
		-\mathrm{Irr}(\mathrm{End}_{K(\!(x)\!)}(V)).
	\]
\end{thm}
\begin{proof}
	This follows from Lemma \ref{invariants2} 
	and Proposition \ref{diffinv}. 
\end{proof}

We end this subsection with the following proposition which relates 
the irregularity and dual Puiseux characteristics.
\begin{prop}
	Let $E_{f,q}$ be an irreducible $K(\!(x)\!)$-connection 
	with the dual Puiseux
	characteristic $(q,p;\beta_{1},\ldots,\beta_{g})$. Then we have
	\[
		\mathrm{Irr}(\mathrm{End}_{K(\!(x)\!)}(E_{f,q}))=
		\sum_{i=1}^{g}(e_{i-1}-e_{i})\beta_{i}.
	\]
\end{prop}
\begin{proof}
	If $p\ge q$, this follows from 
	Lemma \ref{invariants}, Proposition \ref{puiseux} and 
	Theorem \ref{milnumb}.
	However Proposition 4.13 in \cite{Wal} leads us 
	to the following direct proof.
	Let us consider 
	\[
		\mathrm{ord}_{x}\prod_{i=1}^{q-1}(f-f_{i}).
	\]
	Here $f_{i}(x^{\frac{1}{q}})=f(\zeta_{q}^{i}x^{\frac{1}{q}})$.
	Since $i\frac{\beta_{1}}{q}$ is an integer if and only if 
	$i$ is divisible by $\frac{q}{e_{1}}$, 
	we have $\mathrm{ord}_{x}(f-f_{i})=-\beta_{1}/q$ for $i$ such that 
	$\frac{q}{e_{1}}\not|i$ and this happens $q-e_{1}$ times. 
	Similarly, we can see that $\mathrm{ord}_{x}(f-f_{i})=-\beta_{j}/q$ 
	if and only if $i$ is divisible by $\frac{q}{e_{j-1}}$ but 
	not by $\frac{q}{e_{j}}$ and this happens $e_{j-1}-e_{j}$ times.
	Hence we have 
	\[
		\mathrm{ord}_{x}\prod_{i=1}^{q-1}(f-f_{i})=
		-\frac{1}{q}\sum_{i=1}^{g}(e_{i-1}-e_{i})\beta_{i}
	\]
	which induces the required formula.
\end{proof}
\subsection{Local Fourier transforms and birational transforms}
In \cite{Fan} and \cite{Sab}, Fang and Sabbah computed 
explicit structures of local Fourier transforms
of $E_{f,q}$ as we saw in Theorem \ref{locfourier}.
Whereas Fang's computation is  
relatively direct algebraic calculation, Sabbah's is based on the blowing up
technique of plane curve singularities.
The proposition below may connect these two different approaches.
Roughly to say, $\mathcal{F}^{(0,\infty)}$ and $\mathcal{F}^{(\infty,0)}$
can be seen as the blowing up of associated curves 
and $\mathcal{F}^{(\infty,\infty)}$ 
corresponds to the birational transform
\[
	\begin{array}{cccc}
		\sigma_{3}\colon &x&\mapsto &x_{1}^{-1}y_{1}\\
			    &y&\mapsto&y_{1}.
	\end{array}
\]
More precisely, let us consider an irreducible curve germ $C(x,y)$ with 
the good 
parametrization 
\begin{align*}
	x&=t^{m}\alpha(t),\quad \alpha(0)\neq 0,\\
	y&=t^{n}
\end{align*}
and assume $m\le n$. Then define an irreducible curve germ 
$\sigma_{3}^{*}(C(x,y))(x_{1},y_{1})$ so that the good parametrization
of this curve is 
\begin{align*}
	x_{1}&=t^{n-m}\alpha(t)^{-1},\\
	y_{1}&=t^{n}.
\end{align*}

\begin{prop}\label{fourierblowup}
	Let us take $f\in K( \!(x^{\frac{1}{q}})\!)$ so that 
	the image is in $R_{q}^{o}(x)\backslash\{0\}$ and $\ord(f)=-p/q$. 
	\begin{enumerate}
		\item Suppose that $p<q$. 
			Let us take $g\in K(\!(x_{1}^{\frac{1}{q-p}}))$ 
			so that 
			\[
				\mathcal{F}^{(\infty,0)}(E_{f,q})\cong
				E_{\dot{g},q-p},
			\]
			as in Theorem \ref{locfourier} where $\dot{g}=
			g+\frac{p}{2(q-p)}$.
			Then we have
			\[
				C_{g,q-p}(x_{1},-y_{1})=
				\sigma_{2}^{*}(C_{f,q}(x,y)).
			\]
		\item Let us take $g\in K(\!(x_{1}^{\frac{1}{p+q}}))$ 
			so that 
			\[
				\mathcal{F}^{(0,\infty)}(E_{f,q})\cong
				E_{\dot{g},p+q},
			\]
			as in Theorem \ref{locfourier} where 
			$\dot{g}=g+\frac{p}{2(p+q)}$.
			Then we have
			\[
				C_{f,q}(-x,y)=
				\sigma_{2}^{*}(C_{g,p+q}(x_{1},y_{1})).
			\]
		\item Suppose that $p>q$. 
			Let us take $g\in K(\!(x_{1}^{\frac{1}{p-q}}))$ 
			so that
			\[
				\mathcal{F}^{(\infty,\infty)}(E_{f,q})\cong
				E_{\dot{g},p-q},
			\]
			as in Theorem \ref{locfourier} where 
			$\dot{g}=g+\frac{p}{2(p-q)}$.
			Then we have
			\[
				C_{g,p-q}(x_{1},-y_{1})=
				\sigma_{3}^{*}(C_{f,q}(x,y)).
			\]
	\end{enumerate}
\end{prop}
\begin{proof}It may suffices to show $(1)$, since the others are similar.
	The curve germs $C_{f,q}$ and $C_{\tilde{g},q-p}$ have 
	good parametrizations $x=t^{q},\,y=\alpha(t)
	=1/f(x^{\frac{1}{q}})$ and $x_{1}=u^{q-p},\,y_{1}=
	\beta(u)=1/g(x_{1}^{\frac{1}{q-p}})$ respectively.
	By Theorem \ref{locfourier}, we have
	\begin{align*}
		x_{1}&=xf(x^{\frac{1}{q}}),&
		y_{1}&=\beta(u)=\frac{1}{g(x_{1}^{\frac{1}{q-p}})}=
		-\frac{1}{f(x^{\frac{1}{q}})},
	\end{align*}
	that is,
	\begin{align*}
		x_{1}y_{1}&=-x,&
		y_{1}=&-y.
	\end{align*}
	Since each irreducible curve germ is determined 
	by a good parametrization, we are done.
\end{proof}
\subsection{Resolution of ramified irregular singularities}
In the previous section, we saw that local Fourier transforms could be regarded
as the birational transforms of associated curve germs. 
As is well known, singularities
of plane curve germs have a resolution via blowing up. We shall seek an 
analogy of the resolution of singularities for irreducible 
connections via local Fourier 
transforms.

First, let us see how the local Fourier transforms change dual
Puiseux characteristics of connections.
\begin{prop}\label{fourierpuiseux}
	Suppose that an irreducible $E_{f,q}$ has the
	dual Puiseux characteristic $(q,p; \beta_{1},\ldots,\beta_{g})$.
	\begin{enumerate}
		\item If $p<q$, then the dual 
			Puiseux characteristic of 
			$\mathcal{F}^{(\infty,0)}(E_{f,q})$ is 
			\begin{align*}
				(q-p,\beta_{1};\beta_{2},\ldots,\beta_{g})
				&\text{ if }
				(q-p)|\beta_{1},\\
				(q-p,\beta_{1};\beta_{1},\ldots,\beta_{g})&
				\text{ otherwise}.
			\end{align*}
			Here we note that $p=\beta_{1}$	
			under the assumption $p<q$.
		\item The dual Puiseux characteristic of 
			$\mathcal{F}^{(0,\infty)}(E_{f,q})$ is 
			\begin{align*}
				(p+q,\beta_{1};\beta_{1},\ldots,\beta_{g})&
				\text{ if } p=\beta_{1},\\
				(p+q,p;p,\beta_{1},
				\ldots,\beta_{g})&
				\text{ otherwise}.
			\end{align*}
		\item If $p>q$, then the dual Puiseux characteristic of 
			$\mathcal{F}^{(\infty,\infty)}(E_{f,q})$ is 
			\begin{align*}
				(p-q,\beta_{1};\beta_{2},\ldots,\beta_{g})&
				\text{ if }p=\beta_{1}\text{ and }
				(p-q)|\beta_{1},\\
				(p-q,p;\beta_{1},\ldots,\beta_{g})&
				\text{ otherwise}.
			\end{align*}

	\end{enumerate}
\end{prop}
\begin{proof}
	We use the same notation in the proof of Proposition \ref{puiseux}.
	First we note that $E_{h,r}$ and $E_{h+\alpha,r}$ with $
	h\in R_{r}^{o}$ and $\alpha\in K$ have the same dual 
	Puiseux characteristic.
	Thus it is enough to know the Puiseux characteristic of 
	$C_{g,*}$ in Proposition \ref{fourierblowup}.
	Let $x=t^{q},\,y=t^{p}\alpha(t)$, $\alpha\in\mathcal{O}^{*}_{S},\,
	\alpha(0)\neq 0$ 
	be a good parametrization of 
	$C_{f,q}$. As we see in the proof of Proposition \ref{puiseux},
	we have another good parametrization $x=u^{q}\delta(u),\, y=u^{p}$, 
	$\delta\in\mathcal{O}^{*}_{S},\,\delta(0)\neq 0$.
	Then by Proposition \ref{fourierblowup}, $C_{g,q-p}$ has a
	good parametrization $x_{1}=-u^{q-p}\delta(u),\,y_{1}=-u^{p}$.
	Solving $x_{1}=s^{q-p}$, we have another good parametrization
	$x_{1}=s^{q-p},\,y_{1}=s^{p}\gamma(s)$, 
	$\gamma\in\mathcal{O}^{*}_{S},\,\gamma(0)\neq 0$. 
	Thus we have $(1)$.
	We can show $(2)$ in the similar way as $(1)$.

	Let us see $(3)$. We have a good parametrization 
	$x=t^{q},\,y=t^{p}\alpha(t)$, $\alpha\in\mathcal{O}^{*}_{S},\,\alpha(0)
	\neq 0$ of $C_{f,q}$. By solving $u^{p}=t^{p}\alpha(t)$, we have
	another good parametrization $x=u^{q}\delta(u),\,y=u^{p}$ as above.
	Then by Proposition \ref{fourierblowup}, $C_{g,p-q}$ has a
	good parametrization $\xi_{1}=-u^{q-p}\delta(u),\,y_{1}=u^{p}$.
	Here $\xi_{1}=1/x_{1}$. Lemma \ref{wall} allows us to find $\epsilon(u)
	\in\mathcal{O}^{*}_{S},\,\epsilon(0)\neq 0$ such that 
	$x_{1}=u^{p-q}\epsilon(u)$. Finally solving $x_{1}=s^{p-q}$, we 
	have another good parametrization 
	$x_{1}=s^{p-q},\,y_{1}=s^{p}\gamma(s)$, $\gamma\in\mathcal{O}^{*}_{S},
	\,\gamma(0)\neq 0$. Thus we have $(3)$.
\end{proof}

We can obtain $E_{f+ax^{-n},q}$ from $E_{f,q}$ by the tensor product
\[
	E_{f,q}\otimes \left(\mathbb{C}(\!(x)\!),
	\frac{d}{dx}+ax^{-n-1}\right)\cong
	E_{f+ax^{-n},q},
\]
and call this process the {\em addition}.
Suppose that  $E_{f,q}$ is irreducible and has the dual Puiseux characteristic 
$(q,p;\beta_{1},\ldots,\beta_{g})$ with $p\neq \beta_{1}$.
Then applying the addition repeatedly, we can obtain a connection with 
the dual Puiseux characteristic $(q,\beta_{1};\beta_{1},\ldots,\beta_{g})$.

The following theorem determines a necessary and sufficient condition for
an irreducible $E_{f,q}$ to have a resolution of  
ramified irregular singularity via local Fourier transforms.
\begin{thm}\label{main}
	Suppose that an irreducible $E_{f,q}$ has the dual Puiseux 
	characteristic $(q,p;\beta_{1},\ldots,\beta_{g})$. Then 
	we can reduce $E_{f,q}$ to a rank 1 connection by a finite iteration
	of local Fourier
	transforms and additions if and only if we have
	\[
		e_{i-1}\equiv \pm e_{i}\ (\mathrm{mod}\,\beta_{i})
	\]
	for all $i=1,\ldots,g$. Here $e_{0}=q$.
\end{thm}
\begin{proof}
	First we assume that $e_{i-1}\equiv 
	\pm e_{i}\ (\mathrm{mod}\,\beta_{i})$
	for all $i=1,\ldots,g$.
	Applying additions, we may suppose that $E_{f,q}$ has the dual Puiseux
	characteristic $(q,\beta_{1};\beta_{1},\ldots,\beta_{g})$.
	If $q=e_{0}\equiv e_{1}\ (\mathrm{mod}\, \beta_{1})$, then
	$q\ge \beta_{1}$ and Proposition \ref{fourierpuiseux} shows that 
	we can reduce the connection to one with 
	the dual Puiseux characteristic 
	$(e_{1},\beta_{1};\beta_{2},\ldots,\beta_{g})$ by a finite iteration
	of $\mathcal{F}^{(\infty,0)}$. 
	If $q=e_{0}\equiv -e_{1}\ (\mathrm{mod}\, \beta_{1})$,
	then Proposition \ref{fourierpuiseux} shows that 
	we can reduce the connection to one with 
	the dual Puiseux characteristic 
	$(\beta_{1}-e_{1},\beta_{1};\beta_{1},\ldots,\beta_{g})$ 
	by a finite iteration
	of $\mathcal{F}^{(\infty,0)}$. Applying 
	$\mathcal{F}^{(\infty,\infty)}$ to this connection, we have one with 
	$(e_{1},\beta_{1};\beta_{2},\ldots,\beta_{g})$.
	Thus in both cases, we moreover apply the addition and obtain 
	a connection with $(e_{1},\beta_{2};\beta_{2},\ldots,\beta_{g})$.
	We can repeat this process to reduce the connection to a rank 1
	connection with $(e_{g}=1,\beta_{g};\, )$ by our hypothesis.

	Conversely, we assume that $E_{f,q}$ can be reduced to 
	a rank 1 connection  by local Fourier
	transforms and additions. Namely $E_{f,q}$ is constructed 
	from a rank 1 connection by the inversion of local Fourier 
	transforms.
	
	\noindent\textbf{Step 1.}
	Let us start from
	  a rank 1 connection with the dual Puiseux characteristic 
	  $(1,p;\,)$, $p>1$. Then possible inverse transformations are  
	  $(\mathcal{F}^{(\infty,0)})^{-1}$ 
	  and $(\mathcal{F}^{(\infty,\infty)})^{-1}$. 
	  
	  (1-i) Let us apply $(\mathcal{F}^{(\infty,0,)})^{-1}$.
	  Then we have the dual Puiseux characteristic $(1+p,p;p)$.
	  After applying 
	  possible inverse local Fourier transforms, 
	  $(\mathcal{F}^{(\infty,0,)})^{-1}$
	  and $(\mathcal{F}^{(0,\infty)})^{-1}$, 
	  we obtain the  dual Puiseux characteristic 
	  $(q,p;p)$ where $q\equiv 1
	  \ (\mathrm{mod}\,p)$ or go back to $(1,p;\,)$.
	  
	  (1-ii) Let us
	  apply $(\mathcal{F}^{(\infty,\infty)})^{-1}$. 
	  Then the resulting dual Puiseux characteristic 
	  is $(p-1,p;p)$. After applying possible inverse 
	  local Fourier transforms,
	  $(\mathcal{F}^{(\infty,\infty)})^{-1}$ 
	  and $(\mathcal{F}^{(\infty,0)})^{-1}$,
	  we obtain the dual Puiseux characteristic  $(q,p;p)$ where
	  $q\equiv -1\ (\mathrm{mod}\,p)$ or go back to $(1,p;\,)$.
	  
	  \noindent\textbf{Step 2.}
	  Next let us start from the dual Puiseux characteristic $(q,p;p)$ with
	  $q\equiv \pm 1\ (\mathrm{mod}\,p)$ and apply an addition. Then
	  we have the dual Puiseux characteristic $(q,p_1;p)$ with $p_{1}>p$.
	  Set $e_{1}=\mathrm{gcd}(q,p_{1})$. 
	  Now possible inverse transformations are 
	  $(\mathcal{F}^{(\infty,0)})^{-1}$
	  and $(\mathcal{F}^{(\infty,\infty)})^{-1}$.

	  (2-i) Applying $(\mathcal{F}^{(\infty,0)})^{-1}$, we obtain 
	  $(q+p_{1},p_{1};p_{1},p)$. 
	 	  After applying possible inverse local Fourier transforms,
		  $(\mathcal{F}^{(0,\infty)})^{-1}$ 
		  and $(\mathcal{F}^{(\infty,0,)})^{-1}$,
	  we obtain the dual Puiseux characteristic $(q_{1},p_{1};p_{1},p)$
	  with $q_{1}\equiv e_{1}\ (\mathrm{mod}\,p_{1})$ or 
	  go back to $(q,p_{1};p)$.
	  	  
	  (2-ii) Applying $(\mathcal{F}^{(\infty,\infty)})^{-1}$,
	  we obtain
	  $(p_{1}-q,p_{1};p_{1},p)$.
	  After applying possible inverse local Fourier transforms,
	  $(\mathcal{F}^{(\infty,\infty)})^{-1}$ 
	  and $(\mathcal{F}^{(0,\infty)})^{-1}$,
	  we obtain the dual Puiseux characteristic $(q_{1},p_{1};p_{1},p)$
	  with $q_{1}\equiv -e_{1}\ (\mathrm{mod}\,p_{1})$ or 
	  go back to $(q,p_{1};p)$.
	  
	  Our possible transformations are the iteration of these process.
	  Thus the obtained dual Puiseux characteristic  $(p,q;\beta_{1},
	  \ldots,\beta_{g})$ satisfies the required conditions.
\end{proof}
\section{Sequences of total orders and Stokes structures}
In this section we restrict the field $K$ to the field of complex number field
$\mathbb{C}$. 
We denote the ring of convergent power series, 
the field of meromorphic functions near $0$  and 
the ring of convergent power series of $x$ and $y$  by 
$\mathbb{C}\{x\}$, $\mathbb{C}(\!\{x\}\!)$ 
and $\mathbb{C}\{x,y\}$ respectively.
Let us define $k$-th root $x^{\frac{1}{k}}$ of $x$ so that 
it takes a real value when $x$ is real and positive.
Let us consider $f\in \mathbb{C}(\!(x^{\frac{1}{q}})\!)$ whose image 
is in $R_{q}^{o}(x)\backslash\{0\}$ and suppose that $E_{f,q}$ has the 
dual Puiseux characteristic $(q,p;\beta_{1},\ldots,\beta_{s})$.
Then we define 
\[
	\tilde{f}(x^{\frac{1}{q}})=\sum_{i=1}^{g}a_{\beta_{i}}
x^{-\frac{\beta_{i}}{q}}
\]
and $\tilde{f}_{i}(x^{\frac{1}{q}})=\tilde{f}(\zeta_{q}^{i}x^{\frac{1}{q}})$
for $i=1,\ldots,q$, 
where we write $f(x^{\frac{1}{q}})=
a_{p}x^{-\frac{p}{q}}+a_{p-1}x^{-\frac{p-1}{q}}+\cdots$.
If $x$ moves in a small circle 
$S_{\eta}=\{z\in \mathbb{C}\mid 
|z|=\eta\}$, 
the order of sizes of $\mathrm{Re}(\tilde{f}_{i}(x^{\frac{1}{q}}))$
for $i=0,\ldots,q-1$ 
change according to the argument of $x$.
This is one of the reasons of the Stokes phenomenon.
Thus to understand the Stokes phenomenon of the connections 
over $\mathbb{C}(\!\{x\}\!)$ formally isomorphic to $E_{f,q}$,
we study the closed curve
\[
	\mathrm{St}=\left\{(x,y)\,\middle|\, x\in S_{\eta},\,
	y=\mathrm{Re}(\tilde{f}(x^{\frac{1}{q}}))
	\right\}
\]
in this subsection.
This curve can be seen as the projection of the closed curve
\[
	K=\left\{(x,y)\,\middle|\,x\in S_{\eta},\,
		y=\frac{1}{\tilde{f}(x^{\frac{1}{q}})}
	\right\}
\]
by $y\mapsto \mathrm{Re}(1/y)$.
The closed curve $K$ is obtained by restricting $x\in S_{\eta}$ in 
the associated curve germ $C_{\tilde{f},q}(x,y)\in 
\mathbb{C}\{x,y\}$ 
and it is well known that $K$ can be seen as an iterated torus knot.
\subsection{Braids and Plane curve germs}\label{braid}
Let us recall the well known theorem by Brauner 
that irreducible plane curve germs 
describe iterated torus knots around the singular point
. The detail can be found in standard references (\cite{BrKn} for instance). 
Let $C(x,y)\in \mathbb{C}\{x,y\}$ be an irreducible plane curve germ with
the Puiseux characteristic $(m;\beta_{1},\ldots,\beta_{g})$. 
Exchanging $x$ and 
$y$ if necessary, we may assume that $f$ has a good parametrization 
$x=t^{m},\,
y=\sum_{i\ge n}a_{i}t^{i}\in \mathbb{C}\{t\}, (a_{n}\neq 0),$ 
with $n\ge m$.
If we let $x$ run  around a sufficiently small circle 
,then 
\[
	K=\left\{(x,y)\,\middle|\, x\in S_{\eta},
		\ y=\sum_{i\ge n}a_{i}x^{\frac{i}{m}}
	\right\}
\]
describes a knot in a solid torus
\[
	S_{\eta}\times D_{\delta}=
\left\{(\eta e^{\sqrt{-1}s},\epsilon e^{\sqrt{-1}t})\mid 
s,t\in \mathbb{R},\,0\le \epsilon \le \delta\right\}
\]
with a suitable $\delta>0$.
\begin{thm}[K.~Brauner \cite{Brau}]
	The above $K$ is an iterated torus knot of order $g$ and type 
	$(m/e_{1},\beta_{1}/e_{1}),(e_{1}/e_{2},\beta_{2}/e_{2}),\ldots,
	(e_{g-1}/e_{g},\beta_{g}/e_{g})$.
\end{thm}
Now let us recall the construction the iterated torus knot from the good 
parametrization.
First we decompose $y(x)$ as 
$y(x)=\sum_{k=1}^{g}a_{\beta_{k}}x^{\frac{\beta_{k}}{m}}+r(x)$
where $r(x)$ is the term of small oscillations which may be 
ignored.
Thus we focus only on 
$\tilde{y}(x)=\sum_{k=1}^{g}a_{\beta_{k}}x^{\frac{\beta_{k}}{m}}$.
Let us first look at $\tilde{y}^{(1)}=a_{\beta_{1}}x^{\frac{\beta_{1}}{m}}$.
Then 
\[
	K_{1}=\{(x,\tilde{y}^{(1)}(x))\mid x\in S_{\eta}
	\}
\]
is the torus knot of type $(m/e_{1},\beta_{1}/e_{1})$ which can be 
seen as the closed braid of the geometric 
braid $B_{1}$ with the $m/e_{1}$ strings 
\[
	\tilde{y}^{(1)}_{l}(t)=
a_{\beta_{1}}\eta^{\frac{\beta_{1}}{m}}
e^{\sqrt{-1}\frac{\beta_{1}}{m}(t+l)}
\quad (0\le t\le 2\pi),
\]
for $l=1,\ldots,m/e_{1}$.
Here we note that there exists a permutation 
$\tau_{1}\in \mathfrak{S}_{m/e_{1}}$ such that 
\[
	\tilde{y}^{(1)}_{l}(t+2\pi)=\tilde{y}^{(1)}_{\tau_{1}(l)}(t)
\]
for $l=1,\ldots,m/e_{1}$. Here $\mathfrak{S}_{n}$ denotes the symmetric group
of $n$ symbols.

Then next,
$\tilde{y}^{(2)}=a_{\beta_{1}}x^{\frac{\beta_{1}}{m}}
+a_{\beta_{2}}x^{\frac{\beta_{2}}{m}}$ improves the approximation and 
\[
	K_{2}=\{(x,\tilde{y}^{(2)}(x))\mid x\in S_{\eta}
	\}
\]
is the iterated torus knot of order $2$ and type $(m/e_{1},\beta_{1}/e_{1}),
(e_{1}/e_{2},\beta_{2}/e_{2})$.
Indeed, for each $l_{1}=1,\ldots,m/e_{1}$, 
one has $e_{1}/e_{2}$ points 
\[
	\tilde{y}^{(2)}_{l_{1},l_{2}}(t)=
	a_{\beta_{1}}\eta^{\frac{\beta_{1}}{m}}
	e^{\sqrt{-1}\frac{\beta_{1}}{m}(t+l_{1})}
	+a_{\beta_{2}}\eta^{\frac{\beta_{2}}{m}}
e^{\sqrt{-1}\frac{\beta_{2}}{m}(t+l_{2})}\quad
(l_{2}=1,\ldots,e_{1}/e_{2})
\]
in the circle of 
radius $|a_{\beta_{2}}|
\eta^{\frac{\beta_{2}}{m}}$ around the point
$\tilde{y}^{(1)}_{l_{1}}(t)$.
Thus for each $l_{1}$, we have the set $\widehat{B}_{l_{1}}$ 
of the strings 
$\tilde{y}^{(2)}_{l_{1},l_{2}}(t)$
for 
$l_{2}=1,\ldots,e_{1}/e_{2}$.
As we noted above, we can identify $\widehat{B}_{l_{1}}$ and 
$\widehat{B}_{\tau_{1}(l_{1})}$ by substituting $t+2\pi$ for $t$.
Thus it suffices to see $\widehat{B}_{l_1}$ for one $l_{1}\in\{1,\ldots,n/e_{1}\}$.
Then $\widehat{B}_{l_{1}}$ defines a geometric braid $B_{2}$ if 
$t$ runs in the interval $[0,(m/e_{1})2\pi]$ and we have 
the torus knot of type $(e_{1}/e_{2},\beta_{2}/e_{2})$ as the 
closed braid of $B_{2}$.

Then one can repeat this process to refine the approximation and 
obtain the iterated torus knot of the plane curve $C(x,y)$.

\subsection{Representations of sequences of total orders and local moduli of 
differential equations}
For a connection $(\widehat{V},\widehat{\nabla})$ over 
$\mathbb{C}(\!\{x\}\!)$, i.e.,
the pair of finite dimensional $\mathbb{C}(\!\{x\}\!)$-vector 
space $\widehat{V}$
and the $\mathbb{C}$-linear connection $\widehat{\nabla}$, the 
{\em formalization} $(V,\nabla)$ is the connection over 
$\mathbb{C}(\!(x)\!)$
defined by $V=\mathbb{C}(\!(x)\!)
\otimes_{\mathbb{C}(\!\{x\}\!)
}\widehat{V}$ and $\nabla(f\otimes \hat{v})=
\frac{d}{dx}f\otimes \hat{v}+f\otimes \widehat{\nabla}(\hat{v})$ for 
$f\in \mathbb{C}(\!(x)\!)$ and $\hat{v}\in \widehat{V}$.
Let us fix a connection $(V_{0},\nabla_{0})$ over $\mathbb{C}(\!(x)\!)$ and 
consider a $\mathbb{C}(\!\{x\}\!)$-connection 
$(\widehat{V},\widehat{\nabla})$
whose formalization is isomorphic to $(V_{0},\nabla_{0})$.
Let us fix an isomorphism $\xi\colon (V,\nabla)\rightarrow (V_{0},\nabla_{0})$
and 
call $( (\widehat{V},\widehat{\nabla}),\xi)$ a {\em marked pair} formally 
isomorphic to $(V_{0},\nabla_{0})$. 
We say that marked pairs $( (\widehat{V},\widehat{\nabla},\xi))$ and 
$( (\widehat{V}',\widehat{\nabla}'),\xi')$  
are isomorphic if there exists an isomorphism $\hat{u}\colon 
(\widehat{V},\widehat{\nabla})\rightarrow (\widehat{V}',\widehat{\nabla}')$ 
as $\mathbb{C}(\!\{x\}\!)$-connections such that 
$\xi=\xi'\circ u$ where 
$u$ is the isomorphism between the formalizations of them induced by $\hat{u}$.
The isomorphism class of marked pairs formally isomorphic to 
$(V_{0},\nabla_{0})$
is denoted by $\mathfrak{M}( (V_{0},\nabla_{0}))$. 
This local moduli space $\mathfrak{M}( (V_{0},\nabla_{0}))$ is studied by many 
authors (see for instance \cite{BabVar2} and its references) 
and it is known that there exists a one to one correspondence 
from a space of certain unipotent matrices, so called {Stokes matrices},
to $\mathfrak{M}( (V_{0},\nabla_{0}))$ (see Theorem \ref{riemann-hilbert} for
example).

In this subsection we see first that the structure of 
the space of Stokes matrices,
i.e., the local moduli space $\mathfrak{M}( (V_{0},\nabla_{0}))$ 
is determined by a sequence of total orders 
of a finite set. Next we focus on the moduli of $E_{f,q}$ and 
show a structure theorem of the sequence of total
orders by using the iterated torus knot of the associated curve.

\subsubsection{Representations of sequences of total orders}
Let $I$ be a finite set and $<_{0},<_{1},\ldots,<_{h}$ $(h\ge 1)$
a sequence of total orders of $I$. 
We shortly denote the pair of $I$ and 
the sequence by 
\[
	\mathcal{I}=(I,(<_{i})_{i=0,\ldots,h}).
\]
Let us define a representation of
$\mathcal{I}$.
For $\nu=1,\ldots,h$, define subsets of $I\times I$ by 
\[
	\rho_{\nu}=\{(j,k)\in I\times I\mid j\neq k, 
	k<_{\nu-1} j, j<_{\nu} k\}.
\]
Here we note that $\rho_{\nu}$ is {\em anti-symmetric}, i.e.,
$(j,k)\in \rho_{\nu}$ contradicts $(k,j)\in \rho_{\nu}$ and 
{\em transitive}, i.e., $(j,k)\in \rho_{\nu}$ and $(k,l)\in 
\rho_{\nu}$ implies
$(j,l)\in \rho_{\nu}$.
For each $k\in I$, take a finite dimensional $\mathbb{C}$-vector space
$V_{k}$. Then {\em representations of} $\mathcal{I}$ are elements in 
\[
	\mathrm{Rep}(\mathcal{I},(V_{k})_{k\in I})=
	\bigoplus_{\nu=1}^{h}\bigoplus_{(j,k)\in \rho_{\nu}}
	\mathrm{Hom}_{\mathbb{C}}
	(V_{k},V_{j}).
\]
We call $(\mathrm{dim}_{\mathbb{C}}(V_{k}))_{k\in I}\in 
(\mathbb{Z}_{\ge 0})^{I}$
the {\em dimension vector} of $\mathrm{Rep}(\mathcal{I},(V_{k})_{k\in I})$.
For a vector $\alpha=(\alpha_{i})\in (\mathbb{Z}_{\ge 0})^{I}$,
we write
\[
	\mathrm{Rep}(\mathcal{I},\alpha)=
	\mathrm{Rep}(\mathcal{I},(\mathbb{C}^{\alpha_{k}})_{k\in I}).
\]

\subsubsection{Sequence of total orders and that of permutations}
Let us fix a sequence of total orders $\mathcal{I}=
(I,(<_{i})_{i=0,\ldots,h})$. 
For each $i=0,\ldots,h$ let us arrange the elements in $I$, 
\[
	t^{(i)}_{1}<_{i}t^{(i)}_{2}<_{i}\cdots <_{i}t^{(i)}_{n},
\]
and define the bijection
\[
	\begin{array}{cccc}
		\phi_{i}\colon &I&\longrightarrow&\{1,\ldots,n\}\\
										       &t^{(i)}_{k}&\longmapsto&k
	\end{array}.
\]
Here $n$ is the cardinality $\#I$ of $I$.
Then we have a sequence of permutations of $\{1,\ldots,n\}$,
\[
	r_{\nu}=\phi_{\nu}\circ\phi_{\nu-1}^{-1}
	\text{ for }
	\nu=1,\ldots,h.
\]
Conversely if we fix a bijection $\phi_{0}$ from $I$ to 
$\{1,\ldots,n\}$ and a sequence of 
permutations of $\{1,\ldots,n\}$,
\[
	r_{1},\ldots,r_{h},
\]
then we can define a sequence of total orders as follows.
Let us define bijections $\phi_{\nu}\colon I\rightarrow \{1,\ldots,n\}$ by
$\phi_{\nu}=r_{\nu}\circ\phi_{\nu-1}$ for $\nu=1,\ldots,h$.
For each $i=0,\ldots,h$ define the total ordering $<_{i}$ of $I$ as 
the pull back of the natural ordering of $\{1,\ldots,n\}$ by $\phi_{i}$.
Thus we have the following.
\begin{prop}
	Let $I$ be a finite set of the cardinality $n$.
	Then there exists 
	a one to one correspondence between sequences of 
	total orders of $I$ and 
	the pairs of a bijection $\phi_{0}\colon I\rightarrow 
	\{1,\ldots,n\}$ and a sequence of elements in 
	$\mathfrak{S}_{n}$.
\end{prop}

The identity element $\mathrm{id}\in \mathfrak{S}_{n}$ 
may be included in the sequence of permutations $r_{1},\ldots,r_{h}$ 
corresponding to $(I,(<_{i})_{i=0,\ldots,h})$.
It is equivalent to the existence of $i\in \{1,\ldots,h\}$ 
such that $<_{i}$ and 
$<_{i+1}$ define the same order.
Thus we may omit $\mathrm{id}\in \mathfrak{S}_{n}$ in the sequence of permutations and 
call the consequent sequence $r'_{1},\ldots,r'_{h'}$ without $\mathrm{id}\in 
\mathfrak{S}_{n}$ 
the {\em reduced sequence} of permutations.
\begin{df}\normalfont
	Two sequences of total orders $\mathcal{I}$ and $\mathcal{I}'$
	are said to be {\em conjugate} if 
	the corresponding reduced sequcences of permutations are conjugate.
	Namely, let $r_{1},\ldots,r_{h}$ and $r'_{1},\ldots,r'_{h'}$ be
	reduced sequences of permutations corresponding to 
	$\mathcal{I}$ and $\mathcal{I}'$ respectively.
	Then $h=h'$ and there exists $\omega\in \mathfrak{S}_{n}$ such that 
	$r_{\nu}=\omega^{-1}r'_{\nu}\omega$ for all $\nu=1,\ldots,h$.
\end{df}
\subsubsection{Local moduli space and representations of 
sequences of total orders}
We shall construct a sequence of total orders from the Stokes structure
of connections.

Let us consider a $\mathbb{C}(\!(x)\!)$-connection $(V,\nabla)$ with 
a normalized matrix $A(x)\in M(n,\mathbb{C}[x^{-1}])$.
Then it is known that there exists $F\in 
\mathrm{GL}(n,\mathbb{C}(\!(x^{\frac{1}{r}})\!))$ with $r\in \mathbb{Z}_{>0}$
such that 
\begin{multline*}
	FA(x)F^{-1}+\left(\frac{d}{dx}F\right)F^{-1}=\\
	\begin{pmatrix}
		q_{1}I_{m_{1}}&&&\\
		&q_{2}I_{m_{2}}&&\\
		&&\ddots&\\
		&&&q_{s}I_{m_{s}}
	\end{pmatrix}t^{-1}+
	\begin{pmatrix}
		L_{1}&&&\\
		&L_{2}&&\\
		&&\ddots&\\
		&&&L_{s}
	\end{pmatrix}t^{-1}
\end{multline*}
where $t=x^{\frac{1}{r}}$, $q_{i}\in t^{-1}\mathbb{C}[t^{-1}]$ $(q_{i}\neq 
q_{j} \text{ if }i\neq j)$ and 
$L_{i}\in M(m_{i},\mathbb{C})$.
For the finite set 
\[Q_{A}=\{q_{1},q_{2},\ldots,q_{s}\},
\]
we define a sequence of 
total orders as follows.
For $d\in \mathbb{R}$, we write 
\[j<_{d} k \text{\quad  if \quad} 
\mathrm{Re}(a_{0}e^{-\sqrt{-1}l_{0}d})<0
\]
where $q_{j}-q_{k}=a_{0}x^{-l_{0}}
+a_{1}x^{-l_{1}}+\cdots +a_{t}x^{-l_{t}}$ with $l_{0}>l_{1}>
\cdots >l_{t}$, $a_{0}\neq 0$ and 
say $d$ is a {\em Stokes direction} if there 
exist two distinct integers $1\le j,k\le s$ such that these are  
incomparable by $<_{d}$.
Thus we note that if $d$ is not a Stokes direction,
$<_{d}$ defines a total order on $Q_{A}$.

Let $0\le d_{1}<d_{2}<\cdots <d_{h}<2\pi$ be the collection of 
all Stokes directions in $[0,2\pi)$, so called 
{\em basic Stokes directions} (see \cite{BJL2}).
Let us choose $\varepsilon >0$ so that $\tilde{d}_{i}=
d_{i}+\varepsilon <d_{i+1}$ and for 
$i=0,\ldots, h$, where 
$d_{0}$ is the maximum of Stokes directions $d<0$ and 
we formally set $d_{h+1}=2\pi$. 
Then we have the sequence of total orders 
\[
	\mathcal{I}_{A}=(Q_{A},(<_{\tilde{d}_{i}})_{i=0,\ldots,h}).
\]

\begin{rem}
	In the above setting, we see only the basic Stokes directions $d_{i}$
	because there exists
		$\sigma\in \mathfrak{S}_{s}$ such that 
		\[
			q_{\sigma(i)}(e^{2\pi\sqrt{-1}}x)=
			q_{i}(x)
		\]
		for all $i=1,\ldots,s$ and we have 
		\[
			j<_{d} k \text{ if and only if }
			\sigma(j)<_{d+2\pi} \sigma(k)
		\]
		for $d\in \mathbb{R}$.
\end{rem}
Let us associate the representations of $\mathcal{I}_{A}$ and the space 
of certain unipotent  matrices, i.e., so called Stokes matrices.
For each $\nu=1,\ldots,h$, define
\begin{multline*}
	\mathrm{Sto}_{d_{\nu}}(A)=\\
	\left\{(X_{i,j})_{1\le i,j\le s}
		\in \bigoplus_{1\le i,j,\le s}
		\mathrm{Hom}_{\mathbb{C}}(\mathbb{C}^{m_{j}},
		\mathbb{C}^{m_{i}})\,\middle|\,
		X_{i,j}=\begin{cases}
			\mathrm{id}_{\mathbb{C}^{m_{i}}}&\text{ if }i=j\\
			0&\text{ if }(i,j)\notin \rho_{\nu}	
		\end{cases}
	\right\}.
\end{multline*}
Then we have the isomorphism
\[
	\mathrm{Rep}(\mathcal{I}_{A},(m_{i})_{i=1,\ldots,s})\cong 
	\bigoplus_{\nu=1}^{h}\mathrm{Sto}_{d_{\nu}}(A)
\]
as $\mathbb{C}$-vector spaces. 

The following is the direct consequence of Theorem VII and its Remark 2
of \cite{BJL2} (see also \cite{BabVar2,Lod}).
\begin{thm}\label{riemann-hilbert}
	We have a one to one correspondence 
	\[
		\mathrm{Rep}(\mathcal{I}_{A},(m_{i})_{i=1,\ldots,s})\cong
		\bigoplus_{\nu=1}^{h}\mathrm{Sto}_{d_{\nu}}(A)\cong
		\mathfrak{M}( (V,\nabla)).
	\]
\end{thm}

\subsubsection{Sequences of total orders of irreducible connections 
and iterated torus knots of plane curves}
Let us return to our irreducible connection $E_{f,q}$ 
with the dual Puiseux characteristic 
$(q,p;\beta_{1},\ldots,\beta_{g})$. Then we set 
$Q_{E_{f,q}}=\{\tilde{f}_{1},\ldots,\tilde{f}_{q}\}$ and define 
the sequence of total orders $\mathcal{I}_{E_{f,q}}=
(Q_{E_{f,q}},(<_{\tilde{d}_{i}})_{i=0,\ldots,h})$ as in the 
previous subsection.
Recalling that 
\[
	\tilde{f}_{i}(\zeta_{q}x^{\frac{1}{q}})
	=\tilde{f}_{i+1}(x^{\frac{1}{q}})
\]
for $i=1,\ldots,q$ where we set $\tilde{f}_{q+1}=\tilde{f}_{1}$, 
we see that the substitution $x^{\frac{1}{q}}
\mapsto \zeta_{q}x^{\frac{1}{q}}$ defines
the action of $\mathbb{Z}/q\mathbb{Z}$ on $Q_{E_{f,q}}$.

For the latter use, we  introduce the product 
of sequences of total orders
$\mathcal{I}_{1}=(I_{1},(<^{(1)}_{i})_{i=0,\ldots,h^{(1)}})$,
$\mathcal{I}_{2}=(I_{2},(<^{(2)}_{i})_{i=0,\ldots,h^{(2)}})$
with $\#I_{1}=\#I_{2}$.
First suppose that $I_{1}=I_{2}$ and $<^{(1)}_{h^{(1)}}=<^{(2)}_{0}$, then 
the product 
\[
	(I,(\widetilde{<}_{i})_{i=0,\ldots,h^{(1)}+h^{(2)}})=
\mathcal{I}_{1}*\mathcal{I}_{2}
\]
is defined by
\[
	\widetilde{<}_{i}=
	\begin{cases}
		<^{(1)}_{i}&\text{ if }0\le i\le h^{(1)},\\
		<^{(2)}_{i-h^{(1)}}&\text{ if }
		h^{(1)}+1\le i\le h^{(1)}+h^{(2)}.
	\end{cases}
\]
For general cases, find the bijection $\phi\colon I_{1}\rightarrow I_{2}$
such that 
\[
	u<^{(1)}_{h^{(1)}} v \text{ if and only if }
	\phi(u)<^{(2)}_{0}\phi(v)
\]
in $I_{1}$ and define $\phi_{*}(\mathcal{I}_{2})=
(I_{1},(<^{\phi}_{i})_{i=0,\ldots,h^{(2)}})$ so that 
\[
	u <^{\phi}_{k}v\text{ if } \phi(u)<^{(2)}_{k}\phi(v)
\]
in $I_{1}$. Then the product of $\mathcal{I}_{1}$ and $\mathcal{I}_{2}$
 is 
 defined by $\mathcal{I}_{1}*\mathcal{I}_{2}=\mathcal{I}_{1}*
 \phi_{*}(\mathcal{I}_{2})$.

For $k\in \mathbb{Z}_{>0}$
we write $\tilde{f}_{i}\sim_{k}\tilde{f}_{j}$ if 
$\mathrm{deg}_{x^{-\frac{1}{q}}}(\tilde{f}_{i}-\tilde{f}_{j})<k$. 
Let us note that each $\sim_{k}$ preserves orders $<_{d}$ for $d\in 
\mathbb{R}$, i.e., if $\tilde{f}_{i_{1}}\sim_{k}\tilde{f}_{i_{2}}$,
$\tilde{f}_{j_{1}}\sim_{k}\tilde{f}_{j_{2}}$,
$\tilde{f}_{i_{1}}\not\sim_{k}\tilde{f}_{j_{1}}$
 and 
$\tilde{f}_{i_{1}}<_{d}\tilde{f}_{j_{1}}$, then we have 
$\tilde{f}_{i_{\epsilon_{1}}}<_{d}\tilde{f}_{j_{\epsilon_{2}}}$ for 
all $\epsilon_{1},\epsilon_{2}\in\{1,2\}$.
Thus we can consider 
\[
	\mathcal{I}^{(k)}=\mathcal{I}_{E_{f,q}}/\sim_{k}
	=(Q_{E_{f,q}}/\sim_{k},(<_{i})_{i=0,\ldots,h}).
\]
Since $\mathcal{I}^{(k)}$ define the same sequences for 
$\beta_{i}\ge k >\beta_{i+1}$, it suffices to consider 
\[
	\mathcal{I}^{(\beta_{i})},\ i=1,\ldots,g.
\]
We write $I^{(\beta_{i})}=Q_{E_{f,q}}/\sim_{\beta_{i}}$ for short.
Let us note that $I^{(\beta_{i})}$ has the cardinality $q/e_{i}$ 
and the action of 
$\mathbb{Z}/(q/e_{i})\mathbb{Z}$ induced from the $\mathbb{Z}/q\mathbb{Z}$
action on $Q_{E_{f,q}}$.

The natural projections
\[
	\mathcal{I}_{E_{f,q}}=\mathcal{I}^{(\beta_{g})}\xrightarrow[]{\pi_{g}}
	\mathcal{I}^{(\beta_{g-1})}\xrightarrow[]{\pi_{g-1}}
	\cdots\xrightarrow[]{\pi_{2}}
	\mathcal{I}^{(\beta_{1})},
\]
give decompositions
\[
	\mathcal{I}^{(\beta_{i})}=
	\bigsqcup_{a\in I^{(\beta_{i-1})}}
	\mathcal{I}^{(\beta_{i})}_{a},
\]
where $\mathcal{I}^{(\beta_{i})}_{a}=(\pi^{-1}_{i}(a),(<_{i})_{i=0,\ldots,h})$ 
for 
$a\in I^{(\beta_{i-1})}$ and $i=2,\ldots,g$.
This decomposition induces a decomposition of representations of 
$\mathcal{I}_{E_{f,q}}$ as follows. The decomposition below is well known 
as the decomposition of Stokes matrices (see Theorem 8 in \cite{MarRam} or
Proposition I.5.5 in \cite{Lod} for example).
\begin{prop}\label{Stokes decomposition}
	We have a decomposition
	\begin{multline*}
		\mathrm{Rep}(\mathcal{I}_{E_{f,q}},(1)_{i=1,\ldots,q})
		\cong \\
		\mathrm{Rep}(\mathcal{I}^{(\beta_{1})},
		(e_{1})_{i=1,\ldots,q/e_1})
		\oplus
		\bigoplus_{j=2}^{g}
		\bigoplus_{a\in I^{(\beta_{j-1})}}
		\mathrm{Rep}(\mathcal{I}^{(\beta_{j})}_{a},
		(e_{j})_{i=1,\ldots,e_{j-1}/e_{j}}).
	\end{multline*}
\end{prop}
\begin{proof}
	This follows from the decomposition of 
	$M(q,\mathbb{C})$ as below. 
	For each $k=1,\ldots,g$, define 
	\[
		M(q,\mathbb{C})^{(\beta_{i})}=
		\left\{
			(a_{i,j})_{1\le i,j\le q}\in M(q,\mathbb{C})\,
			\middle|\,
			a_{i,j}=0\text{ if }
			\mathrm{deg}_{x^{-\frac{1}{q}}}(\tilde{f}_{i}-
			\tilde{f}_{j})\neq\beta_{k}
		\right\}.
	\]
	Then we have a decomposition
	\begin{equation*}
		M(q,\mathbb{C})=
		\left\{\mathrm{diag}(a_{1},\ldots,a_{q})\mid a_{i}\in \mathbb{C}
		\right\}
		\oplus
		\bigoplus_{i=1}^{g}M(q,\mathbb{C})^{(\beta_{i})}
	\end{equation*}
	as a $\mathbb{C}$-vector space.
\end{proof}
For each $i=2,\ldots,g$ let us fix $o\in I^{(\beta_{i-1})}$
as the image of $\tilde{f}_{1}\in Q_{E_{f,q}}$ and define 
a product of $\mathcal{I}_{a}^{(\beta_{i})}$ for 
$a\in I^{(\beta_{i-1})}$ by
\[
	\widetilde{\mathcal{I}}^{(\beta_{i})}
	=(\widetilde{I}^{(\beta_{i})},
	(\widetilde{<}_{j})_{j=0,\ldots,h_{i}})=
	\mathcal{I}^{(\beta_{i})}_{o}*
	\mathcal{I}^{(\beta_{i})}_{e(o)}*
	\mathcal{I}^{(\beta_{i})}_{e^{2}(o)}*
	\cdots*
	\mathcal{I}^{(\beta_{i})}_{e^{q/e_{i-1}-1}(o)},
\]
and for $i=1$ set  
$\widetilde{\mathcal{I}}^{(\beta_{1})}=\mathcal{I}^{(\beta_{1})}$.
Here $e\in \mathbb{Z}/(q/e_{i-1})\mathbb{Z}$ is the image of $1\in \mathbb{Z}$.
Let us note that there exists the natural isomorphism
\[
	\mathrm{Rep}(\widetilde{\mathcal{I}}^{(\beta_{j})},
	(e_{j})_{i=1,\ldots,e_{j-1}/e_{j}})
	\xrightarrow[]{\sim}
	\bigoplus_{a\in I^{(\beta_{j-1})}}
	\mathrm{Rep}(\mathcal{I}^{(\beta_{j})}_{a},
	(e_{j})_{i=1,\ldots,e_{j-1}/e_{j}})
\]
as $\mathbb{C}$-vector spaces for each $j=2,\ldots,g$. Thus by Proposition 
\ref{Stokes decomposition} we have 
\[
	\mathrm{Rep}(\mathcal{I}_{E_{f,q}},(1)_{i=1,\ldots,q})
	\cong 
	\bigoplus_{j=1}^{g}\mathrm{Rep}(\widetilde{\mathcal{I}}^{(\beta_{j})},
	(e_{j})_{i=1,\ldots,e_{j-1}/e_{j}}).
\]
The following is the main theorem of this subsection which shows that 
the structure of $\mathcal{I}_{E_{f,q}}$ is determined by the 
dual Puiseux characteristic.
This can be seen as 
an analogy of plane curve germs for which Puiseux characteristics 
are topological invariants of knot structures, namely, if two curve 
germs have the same Puiseux characteristic, then the knots of them
are isotopic.
\begin{thm}\label{iterated braid}
	For each $i=1,\ldots,g$,
	there exists 
	$\omega\in \mathfrak{S}_{e_{i-1}/e_{i}}$ and
$\widetilde{\mathcal{I}}^{(\beta_{i})}$ defines
	the sequence of elements in $\mathfrak{S}_{e_{i-1}/e_{i}}$,
	\[
		(s^{\omega}_{1}s^{\omega}_{2}\cdots 
		s^{\omega}_{e_{i-1}/e_{i}-1})^{\beta_{i}/e_{i}}
	\]
	where we omit the identity element $\mathrm{id}
	\in \mathfrak{S}_{e_{i-1}/e_{i}}$, 
	$s_{j}$ are the transpositions $(j,j+1)$ and 
	$s^{\omega}_{j}=\omega^{-1}s_{j}\omega$.
\end{thm}
\begin{proof}
	Let us proceed as the argument in the subsection \ref{braid}. 
	We write
	$\tilde{f}(x^{\frac{1}{q}})
	=\sum_{k=1}^{g}a_{\beta_{k}}x^{-\frac{\beta_{k}}{q}}$.
	Let us first look at $\tilde{f}^{(1)}(x)=
a_{\beta_{1}}x^{-\frac{\beta_{1}}{q}}$.
	If $x$ moves in $S_{\eta}$ for a sufficiently small $\eta>0$, 
	then $\tilde{f}^{(1)}(x)$ moves along a
	small circle centered at  $\infty$.
	The geometric braid $B_{1}$ with the $q/e_{1}$ strings
	\[
		\tilde{f}^{(1)}_{l}(t)=a_{\beta_{1}}\eta^{-\frac{\beta_{1}}{q}}
		e^{-\sqrt{-1}\frac{\beta_{1}}{q}(t+l)}
		\quad (0\le t\le 2\pi)
	\]
	for $l=1,\ldots,q/e_{1}$ define the torus knot of type 
	$(q/e_{1},\beta_{1}/e_{1})$ as the closed braid of $B_{1}$.
	As is well known, 
	if we number the strings in $B_{1}$ suitably, we have 
	the braid words 
	\[
		(\sigma_{1}\sigma_{2}\cdots\sigma_{q/e_{1}-1})
		^{\beta_{1}/e_{1}},
	\]
	where $\sigma_{i}$ are standard generators of the braid
	group $\mathcal{B}_{q/e_{1}}$ on $q/e_{1}$ strings.
	On the other hand, let us consider the
	finite set 
	\[
		\mathfrak{I}^{(\beta_{1})}
		=\{\mathrm{Re}(\tilde{f}^{(1)}_{1}(t)),\ldots,
		\mathrm{Re}(\tilde{f}^{(1)}_{q/e_{1}}(t))\}.
	\]
	Here if  $t$ moves from $0$ to $2\pi$ then $\mathfrak{I}^{(\beta_{1})}$
	defines a sequence of total orders which is nothing but 
	$\mathcal{I}^{(\beta_{1})}$ by a suitable 
	identification $\mathfrak{I}^{(\beta_{1})}\cong I^{(\beta_{1})}$.
	Since this can be seen as  the projection of $B_{1}$ by 
	$\tilde{f}^{(1)}_{l}(t)\mapsto \mathrm{Re}(\tilde{f}^{(1)}_{l}(t))$,
	thus the sequence of total orders defines the sequence of permutations
	\[
		(s_{1}s_{2}\cdots s_{q/e_{1}-1})^{\beta_{1}/e_{1}}
	\]
	as required.

	Next let us fix $j\in \{2,\ldots,g\}$ and consider 
\[
	\tilde{f}^{(j)}(x)=
	\sum_{k=1}^{j}a_{\beta_{k}}x^{-\frac{\beta_{k}}{q}}.
\]
For $1\le k\le j$ and  $1\le l_{k}\le e_{k-1}/e_{k}$ let us define
\[
	\tilde{f}^{(k)}_{l_{1},\ldots,l_{k}}(t)=
	\sum_{i=1}^{k}a_{\beta_{i}}\eta^{-\frac{\beta_{i}}{q}}
	e^{-\sqrt{-1}\frac{\beta_{i}}{q}(t+l_{i})}.
\]
Then as we see in the subsection \ref{braid}, 
for a fixed $(l_{1},\ldots,l_{j-1})$ and $l_{j}=1,\ldots,e_{j-1}/e_{j}$, 
one has the $e_{j-1}/e_{j}$ points 
$f^{(j)}_{l_{1},\ldots,l_{j}}(t)$ 
in the circle around the point $\tilde{f}^{(j-1)}_{l_{1},\ldots,l_{j-1}}(t)$.
Moreover the strings
\[
	\tilde{f}^{(j)}(t)_{l_{1},\ldots,l_{j-1},l_{j}}(t)\quad
	(t\in [0,(q/e_{j-1})2\pi])
\]
for $l_{j}=1,\ldots,e_{j-1}/e_{j}$ define a geometric braid $B_{j}$ and 
we have a torus knot of type $(e_{j-1}/e_{j},\beta_{j}/e_{j})$ as 
the closed braid of $B_{j}$. Thus $B_{j}$ defines the braid words
\[
	(\sigma_{1}\sigma_{2}\cdots \sigma_{e_{j-1}/e_{j}-1})^{\beta_{j}/e_{j}}.
\]
By the same argument as above,
if $t$ moves form $0$ to $(q/e_{j-1})2\pi$ then  
\[
	\mathfrak{J}^{\beta_{1}}_{l_{1},\ldots,l_{j-1}}=
	\left\{
		\mathrm{Re}(\tilde{f}^{(j)}_{l_{1},\ldots,l_{j-1},1}(t)),
		\ldots,
		\mathrm{Re}(\tilde{f}^{(j)}_{l_{1},\ldots,
		l_{j-1},e_{j-1}/e_{j}}(t))
	\right\}
\]
defines a sequence of total orders which induces the sequence of permutations
\[
	(s_{1}s_{2}\cdots s_{e_{j-1}/e_{j}-1})^{\beta_{j}/e_{j}}.
\]
Meanwhile this sequence of total orders can be identified with 
$\widetilde{\mathcal{I}}^{(\beta_{j})}$ by a suitable identification
$\mathfrak{J}^{(\beta_{j})}_{l_{1},\ldots,l_{j-1}}\cong 
\widetilde{I}^{(\beta_{j})}$. Thus we are done.

\end{proof}
Thus if we fix a dual Puiseux characteristic $(q,p;\beta_{1},\ldots,\beta_{g})$,
then the conjugacy classes of $\widetilde{\mathcal{I}}^{(\beta_{i})}$,
$i=1,\ldots,g$, are 
determined.
\begin{cor}\label{conjugate}
	Let $E_{f,q}$ and $E_{f',q}$ be irreducible 
	$\mathbb{C}(\!(x)\!)$-connections 
	with the same dual Puiseux characteristic
	$(q,p;\beta_{1},\ldots,\beta_{g})$ and 
	set $\mathcal{I}=\mathcal{I}_{E_{f,q}}$, $\mathcal{I}'=
	\mathcal{I}_{E_{f',q}}.$
	Then the  
	sequences of 
	total orders $\widetilde{\mathcal{I}}^{(\beta_{i})}$ and 
	$\widetilde{\mathcal{I}'}^{(\beta_{i})}$ defined as above are 
	conjugate for each $i=1,\ldots,g$.
\end{cor}

\end{document}